\documentclass[12pt]{article}%
\usepackage{rotating}
\usepackage{amsmath}
\usepackage{amsfonts}
\usepackage{amssymb}
\usepackage{graphicx}
\usepackage{placeins}
\usepackage{geometry}
\usepackage{boxedminipage}
\usepackage[usenames]{color}
\usepackage[bookmarks,colorlinks=true,pdfstartview=FitH,citecolor=darkred,linkcolor=darkred,urlcolor=darkred]%
{hyperref}
\usepackage{natbib}
\usepackage{cleveref}
\usepackage{caption}%
\providecommand{\U}[1]{\protect\rule{.1in}{.1in}}
\definecolor{darkred}{rgb}{0.02,0.14,0.3}
\newtheorem{theorem}{Theorem}

\newtheorem{algorithm}{Algorithm}

\newtheorem{lemma}{Lemma}

\newenvironment{proof}[1][Proof]{\noindent\textbf{#1.} }{\ \rule{0.5em}{0.5em}}
\crefname{proposition}{Proposition}{Propositions}
\Crefname{proposition}{Proposition}{Propositions}
\crefname{cor}{Corollary}{Corollaries}
\Crefname{cor}{Corollary}{Corollaries}
\crefname{table}{Table}{Tables}
\Crefname{table}{Table}{Tables}
\crefname{figure}{Figure}{Figures}
\Crefname{figure}{Figure}{Figures}
\crefname{section}{Section}{Sections}
\Crefname{section}{Section}{Sections}
\crefname{appendix}{Appendix}{Appendices}
\Crefname{appendix}{Appendix}{Appendices}
\crefname{equation}{equation}{equations}
\Crefname{equation}{Equation}{Equations}
\creflabelformat{equation}{#2\textup{(#1)}#3}
\begin{document}

\title{Alternating cyclic extrapolation methods for optimization algorithms}
\author{Nicolas Lepage-Saucier\thanks{Concordia University\indent Corresponding
address: nicolas.lepagesaucier@concordia.ca}}
\date{August 2021}
\maketitle

\begin{abstract}
This article introduces new acceleration methods for fixed-point iterations.
Extrapolations are computed using two or three mappings alternately and a new
type of step length is proposed with good properties for nonlinear
applications. The methods require no problem-specific adaptation and are
especially efficient in high-dimensional contexts. Their computation uses few
objective function evaluations, no matrix inversion and little extra memory. A
convergence analysis is followed by eight applications including gradient
descent acceleration for constrained and unconstrained optimization.
Performances are on par with or better than competitive alternatives. The 
algorithm is available as the Julia package SpeedMapping.jl.

\medskip

%


\end{abstract}

\section{Introduction}

Let $F:%
\mathbb{R}
^{n}\rightarrow%
\mathbb{R}
^{n}$ denote a mapping which admits continuous, bounded partial derivatives.
Finding a fixed point of $F$, $x^{\ast}:F(x^{\ast})=x^{\ast}$, is the basis of
countless numerical applications in disciplines like statistics, computer
science, physics, biology, and economics, driving the development of many
general and domain-specific iterative methods. For reviews and insightful
comparisons of the most important ones, see notably \cite{Jbilou2000},
\cite{Ramiere2015}, \cite{Brezinski2018} and the textbook by
\cite{Brezinski2020}.

One of these methods is a vector version of Aitken's $\Delta^{2}$ process
usually attributed to \cite{Lemarechal1971} but also discovered by
\cite{Irons1969} and \cite{Jennings1971}. Following Aitken's notation, define
$\Delta x=F(x)-x$ and note that at the fixed point of $F$, $\Delta x^{\ast
}=\mathbf{0}$. Lemar\'{e}chal's method may be interpreted as a simplified
quasi-Newton method for finding a root of $\Delta x$:%
\begin{equation}
x_{k+1}=x_{k}-M_{k}^{-1}\Delta x_{k}\label{eq:quasi-Newton}%
\end{equation}
where $M_{k}$ is the approximation of the Jacobian of $\Delta x_{k}$. It takes
the simple form $M_{k}=s_{k}^{-1}I$, where $s_{k}$ is a scalar and, contrary
to other quasi-Newton methods, it ignores off-diagonal elements. The Jacobian
is approximated by the secant method using two consecutive evaluations of the
mapping $F$:%
\[
\frac{d\Delta x_{k}}{dx_{k}}\approx\frac{\Delta^{2}x_{k}}{\Delta x_{k}%
}\text{.}%
\]
where $\Delta^{2}x_{k}=\Delta F(x_{k})-\Delta x_{k}=F(F(x_{k}))-2F(x_{k}%
)+x_{k}$, and, in general, $\Delta^{p}x=\Delta^{p-1}F(x)-\Delta^{p-1}x$ for
$p\in%
\mathbb{N}
^{+}$. The step length is%
\begin{equation}
s_{k}=\arg\min_{s}\left\Vert s^{-1}I-\frac{\Delta^{2}x_{k}}{\Delta x_{k}%
}\right\Vert ^{2}=\frac{\langle\Delta^{2}x_{k},\Delta x_{k}\rangle}{\left\Vert
\Delta^{2}x_{k}\right\Vert ^{2}}\label{s}%
\end{equation}

where $\left\Vert y\right\Vert _{p}$ is the $p$-norm of a vector $y$,
$\left\Vert y\right\Vert =\left\Vert y\right\Vert _{2}=\sqrt{y^{\intercal}y}$
is the 2-norm of a vector $y$, and $\langle y,z\rangle=y^{\intercal}z$ is the
inner product of vectors $y$ and $z$. Substituting $s_{k}$ in
(\ref{eq:quasi-Newton}), an iteration of Lemar\'{e}chal's method is%

\begin{equation}
x_{k+1}=x_{k}-\frac{\langle\Delta^{2}x_{k},\Delta x_{k}\rangle}{\left\Vert
\Delta^{2}x_{k}\right\Vert ^{2}}\Delta x_{k}\text{.}\label{eq:Lemarechal}%
\end{equation}

Early on, this simple way of computing the step length proved to be generally
faster and more stable than comparable techniques (see \cite{Henrici1964} or
\cite{Macleod1986}). It was later improved upon by \cite{Barzilai1988}. The
Barzilai-Borwein (BB) method requires a single mapping per iteration by using
the Cauchy step length of the previous iteration. It spawned a rich line of
research in gradient-based optimization for linear problems, a branch of which
investigates the link between the optimal step size and the Hessian spectral
properties (see \cite{Birgin2014} for a good review), to which this paper is
also relevant.

As will be argued, this choice of $s_{k}$ may cause slow convergence if the Jacobian of $\Delta x$ has a wide spectrum. To avoid this drawback, new step lengths will be introduced to target specific deviations of $x$ from its fixed point and considerably speed-up convergence over time. 

For exposition, let us consider a system of linear equations%
\[
Ax=b
\]
where $A\in%
\mathbb{R}
^{n\times n}$ and $b\in%
\mathbb{R}
^{n}$. The solution of the system also constitutes the minimizer of the
quadratic function $f(x)=\frac{1}{2}x^{\intercal}Ax-x^{\intercal}b$ with
gradient $\triangledown f(x)=Ax-b$ and Hessian $A$. To avoid the need for a
change of coordinates, assume $A=diag(\lambda_{1}...\lambda_{n})$ with
positive entries and $m\leq n$ distinct eigenvalues with the smallest and
largest labeled $\lambda_{\min}$ and $\lambda_{\max}$, respectively.

The problem may be formulated as finding the fixed point to the mapping
\[
F(x)=x-(Ax-b)\text{,}%
\]

with unique solution $x^{\ast}$ representing a fixed point of $F$ at which
$Ax^{\ast}=b$. To study the convergence of the method, define an error as
$e=x-x^{\ast}$. Direct computation gives $\Delta x=-(Ax-b)=-Ae$ and, in
general, $\Delta^{p}x=(-A)^{p}e$. Since $A$ is diagonal, $s_{k}$ may be
expressed as%
\begin{equation}
s_{k}=\frac{\langle A^{2}e_{k},-Ae_{k}\rangle}{\left\Vert A^{2}e_{k}%
\right\Vert ^{2}}=-\frac{e_{k}^{\intercal}A^{3}e_{k}}{e_{k}^{\intercal}%
A^{4}e_{k}}=-\frac{\sum_{i=1}^{n}\left(  e_{k}^{(i)}\lambda_{i}^{2}\right)
^{2}\frac{1}{\lambda_{i}}}{\sum_{i=1}^{n}\left(  e_{k}^{(i)}\lambda_{i}%
^{2}\right)  ^{2}}\text{.}\label{s2}%
\end{equation}
We can write (\ref{eq:Lemarechal}) in terms of errors and obtain for the
linear system
\[
e_{k+1}=\left(  I+s_{k}A\right)  e_{k}\text{.}%
\]
In particular, the $j^{th}$ error component may be individually expressed as%

\begin{equation}
e_{k+1}^{(j)}=\left(  1+s_{k}\lambda_{j}\right)  e_{k}^{(j)}\text{,\qquad
}j=1,...,n\text{.} \label{eq:error_component}%
\end{equation}

As shown by (\ref{eq:error_component}), if $s_{k}$ was somehow set exactly
equal to $-\frac{1}{\lambda_{j}}$, the error $e_{k+1}^{(j)}$ would be
perfectly annihilated. It would also remain zero for all subsequent
iterations, regardless of $s_{k+1},s_{k+2},...$. Since $A$ has $m$ distinct
eigenvalues, all error components could be successively reduced to zero in $m$
iterations. Conversely, from a starting point with at least one positive error
component $e_{0}^{(j)}$ for each distinct eigenvalue, $m$ is also the minimum
number of steps necessary to annihilate all $e^{(j)}$ exactly. Of course, as
shown by (\ref{s2}), $s_{k}\in\lbrack-\lambda_{\min}^{-1},-\lambda_{\max}%
^{-1}]$ is a weighted average of all negative inverse eigenvalues. As a
result, all $e^{(j)}$ are imperfectly reduced simultaneously at each iteration
and convergence may be slow if $A$ has a wide spectrum. 

Instead, targeting specific error components may be more efficient. To do so, note
that at the fixed point $x^{\ast}$, any higher-order difference $\Delta^{p}x$
is zero, not only $\Delta x$. The rate of change of $\Delta^{p}x$ for
$p>1$ also carries useful information on the location of the fixed point. Define 
a step length of order $p$ as%

\begin{equation}
s^{(p)}=\frac{\langle\Delta^{p}x,\Delta^{p-1}x\rangle}{\left\Vert \Delta
^{p}x\right\Vert ^{2}}\label{eq:sp}%
\end{equation}

where $s^{(2)}\equiv s$ defined in (\ref{s}) is a special case. In particular,
a \textquotedblleft cubic\textquotedblright\ step length $s^{(3)}$ may
be interpreted as the change in $\Delta^{3}x$ following a change of
$\Delta^{2}x$. In the linear example, $s_{k}^{(3)}$ would be%
\[
s_{k}^{(3)}=-\frac{e_{k}^{\intercal}A^{5}e_{k}}{e_{k}^{\intercal}A^{6}e_{k}%
}=-\frac{\sum_{i=1}^{n}\left(  e_{k}^{(i)}\lambda_{i}^{3}\right)  ^{2}\frac
{1}{\lambda_{i}}}{\sum_{i=1}^{n}\left(  e_{k}^{(i)}\lambda_{i}^{3}\right)
^{2}}\text{.}%
\]
Using $s_{k}^{(3)}$ for approximating the Jacobian of $\Delta x_{k}$ puts
more weights on eigenvalues of larger magnitudes and should annihilate the 
associated errors more aggressively. After such step $k$, a squared step $k+1$
should target errors associated with smaller eigenvalues more precisely
since the error associated with larger ones would have a lighter
weight in $s_{k+1}^{(2)}$. The algorithm can then continue alternating
between cubic and squared extrapolations. 

Schemes based on squared iterations
alone rely on the hope that $\Delta x$ changes at similar rates for all
components of $e$ which is false if $A$ has a wide spectrum. Starting with 
a cubic extrapolation selects a path in
parameter space where the following squared extrapolation suffers less from
this source of error. A single cubic iteration may not converge
faster but alternating between different orders may dynamically enhance
convergence over time.

Of course, computing higher-order step lengths requires more
mappings.\footnote{It may explain why few acceleration schemes requiring
third-order or even higher-order differences have been put forth. Notable
exceptions are \cite{Marder1970}, \cite{Lebedev1995} and \cite{Brezinski1998}%
.} Fortunately, this extra compute cost is mitigated by the ability to use
$s^{(p)}$ for many extrapolations thanks to the important idea of cycling
(\cite{Friedlander1999} and \cite{Raydan2002}). Raydan and Svaiter noted that
the iterates of the BB method are themselves mappings with the same fixed
point $x^{\ast}$ and that the same step length could be used for a second
extrapolation with little extra computational cost. To understand the idea,
denote an intermediate series $y_{k,(i)}$ constructed from Lemar\'{e}chal's
method applied to $x_{k},F(x_{k}),F(F(x_{k})),...$: $y_{k,(1)}=x_{k}%
-s_{k}\Delta x_{k}$, $y_{k,(2)}=F(x_{k})-s_{k}\Delta F(x_{k})$, $...$ . Note
that the series $y_{k,(1)},y_{k,(2)},...$ also converges to $x^{\ast}$ and can
therefore be extrapolated in the same fashion and using the same step length:%
\[
y_{k+1}=y_{k,(1)}+s_{k}(y_{k,(2)}-y_{k,(1)})
\]

After substitution, the squared cyclic extrapolation can be written as%
\[
x_{k+1}=y_{k+1}=x_{k}-2s_{k}\Delta x_{k}+s_{k}^{2}\Delta^{2}x_{k}\text{.}%
\]

This Cauchy-Barzilai-Borwein (CBB) method has good convergence properties. It
was successfully adapted to a variety of nonlinear contexts by Varadhan and
Roland (2004\nocite{Roland2004}, 2005\nocite{Roland2005} and
2008\nocite{Roland2008}), notably to accelerate the expectation maximization
(EM) algorithm (\cite{Ortega1970}, \cite{Dempster1977}) under the label SQUAREM.

For a $p$-order step, cycling may be performed $p$ times in the same recursive
manner. In the linear system previously defined, with a step $s_{k}^{(p)}$ and
$p$-order cycling, the error $e_{k+1}$ becomes%
\[
e_{k+1}=\left(  I+s_{k}^{(p)}A\right)  ^{p}e_{k}\text{.}%
\]
By cycling and alternating between different extrapolation orders, the
proposed scheme is best understood as an alternating cyclic extrapolation
method (ACX). Two specific schemes will be studied empirically: ACX$^{3,2}$,
which alternates between cubic ($p=3$) and squared ($p=2$) extrapolations, and
ACX$^{3,3,2}$ which performs two cubic extrapolations before one squared
extrapolation. They display good overall empirical properties, but higher order extrapolations or other sequences could be considered in specific contexts. For simplicity, ACX will be shorthand for any of these two methods in the rest of the paper. 

The next section describes the algorithm formally and introduces a new type of
step length with good properties for ACX in nonlinear contexts. Section
\ref{sec:proof} establishes Q-linear convergence\ of ACX for linear mappings
in a suitable norm and discusses its convergence properties in nonlinear
contexts. Section \ref{sec:Implementation} discusses stability issues and
implementation details, including an adaptation of ACX to gradient descent
acceleration with dynamic adjustment of the descent step size.

ACX has many advantages. It can be used \textquotedblleft
as-is\textquotedblright\ for a wide variety of problems without extra
specialization and close to no tuning. It requires few to no objective
function evaluations, no matrix inversion, and little extra memory. It stands to
shine in high dimensional applications which are ever more prevalent with the
proliferation of large datasets and sparse data structures.

These advantages are on display in Section \ref{sec:Applications}.
ACX is applied to gradient descent acceleration for various constrained and
unconstrained optimizations including 96 problems from the CUTEst collection
(see \cite{Bongartz1995}). It compares favorably to popular methods like the
limited-memory Broyden--Fletcher--Goldfarb--Shanno algorithm (L-BFGS) (see
\cite{Liu1989} and \cite{Nocedal2006}) and the nonlinear conjugate gradient
method (N-CG) proposed by \cite{Hager2006}. It also performs well for various
other fixed-point iterations compared to competitive alternatives like
the quasi-Newton acceleration of \cite{Zhou2011}, the objective acceleration
approach of \cite{Riseth2019}, the Anderson acceleration version of
\cite{Henderson2019}, and other domain-specifics algorithms.

ACX is available as the Julia package SpeedMapping.jl.

\section{Alternating-orders cyclic extrapolations%
\label{sec:ACX}%
}

A $p$-order ACX iteration may be synthesized as:%
\begin{equation}
x_{k+1}=\sum_{i=0}^{p}\tbinom{p}{i}(\sigma_{k}^{(p)})^{i}\Delta^{i}%
x_{k}\text{\qquad}p\geq2\label{eq:ACX}%
\end{equation}
where $\Delta^{0}x_{k}=x_{k}$, $\binom{p}{i}=\frac{p!}{i!(p-i)!}$ is a
binomial coefficient, and $\sigma_{k}^{(p)}=|s_{k}^{(p)}|\geq0$ is the
absolute value of the step length (\ref{eq:sp}). 

Step lengths other than $s$ (\ref{s}) (sometimes referred to as $s^{BB2}$),
have been suggested in the literature. \cite{Barzilai1988} suggested
$s^{BB1}=\left\Vert \Delta x\right\Vert ^{2}/\langle\Delta^{2}x,\Delta
x\rangle$. \cite{Roland2008} introduced $s^{RV}=-\sqrt{s^{BB1}s^{BB2}%
}=-\left\Vert \Delta x\right\Vert /\left\Vert \Delta^{2}x\right\Vert $ (this
author's notation), commenting that for nonlinear mappings, $s^{BB1}$ can
compromise stability since the denominator $\langle\Delta^{2}x,\Delta
x\rangle$ may be close to zero or even positive. Similarly, $s^{BB2}$ may be
problematic if $\langle\Delta^{2}x,\Delta x\rangle$ is positive and
$\left\Vert \Delta^{2}x\right\Vert ^{2}$ is small. In contrast, $s^{RV}$ has a
guaranteed negative sign for better overall stability. In the case of
ACX$^{3,2}$ and ACX$^{3,3,2}$ however, an even better choice turns out to be
$-\sigma^{(p)}\equiv-|s^{(p)}|$. It is a simple way of avoiding wrong signs
while providing better overall convergence. Other options could be explored given the growing literature on
optimal step sizes for descent algorithms (see for instance \cite{Dai2019} for
a recent contribution).

The ACX algorithm is formalized as follows.%

\noindent\makebox[\linewidth]{\rule{\textwidth}{0.4pt}}%

\vspace*{-0.3cm}

\begin{algorithm}%
\label{algo:ACX}%
Input: a mapping $F:%
\mathbb{R}
^{n}\rightarrow%
\mathbb{R}
^{n}$, a starting point $x_{0}\in%
\mathbb{R}
^{n}$, and a vector of orders $(o_{1},...,o_{P})$ with $o_{j}\in\{2,3\}$.
\end{algorithm}

\vspace*{-0.5cm}%

\noindent\makebox[\linewidth]{\rule{\textwidth}{0.4pt}}%

\smallskip%

\renewcommand{\arraystretch}{1.2}%
%

\noindent
\begin{tabular}
[c]{p{0.6cm}l}%
1 & for $k=0,1,2,...$ until convergence\\
2 & \qquad$p_{k}=o_{(k\operatorname{mod}P)+1}$\\
\multicolumn{1}{l}{3} & \qquad$\Delta^{0}=x_{k}$\\
\multicolumn{1}{l}{4} & \qquad$\Delta^{1}=F(x_{k})-x_{k}$\\
\multicolumn{1}{l}{5} & \qquad$\Delta^{2}=F^{2}(x_{k})-2F(x_{k})+x_{k}$\\
\multicolumn{1}{l}{6} & \qquad If $p_{k}=3$: $\Delta^{3}=F^{3}(x_{k}%
)-3F^{2}(x_{k})+3F(x_{k})-x_{k}$\\
\multicolumn{1}{l}{7} & \qquad$\sigma_{k}^{(p_{k})}=|\langle\Delta^{p_{k}%
},\Delta^{p_{k}-1}\rangle|/\left\Vert \Delta^{p_{k}}\right\Vert ^{2}$\\
\multicolumn{1}{l}{8} & \qquad$x_{k+1}=\sum_{i=0}^{p_{k}}\tbinom{p_{k}}%
{i}(\sigma_{k}^{(p_{k})})^{i}\Delta^{i}$\\
\multicolumn{1}{l}{9} & end for
\end{tabular}
%

\renewcommand{\arraystretch}{1}%
%

\noindent\makebox[\linewidth]{\rule{\textwidth}{0.4pt}}%

\vspace*{0.1cm}%

\noindent
Note: To improve global convergence, constraints on $\sigma_{k}^{(p_{k})}$ may
be imposed at step 7 and bounds checking on $x_{k+1}$ at step 8. Also, a
preliminary mapping before step 3 may improve the convergence of certain
algorithms. See Section \ref{sec:Implementation} for implementation detail.

\vspace*{-0.3cm}%

\noindent\makebox[\linewidth]{\rule{\textwidth}{0.4pt}}%

\vspace*{0.2cm}

Section \ref{sec:proof} shows that for linear maps, Algorithm \ref{algo:ACX}
produces a sequence $x_{k}$ converging to a fixed point of $F$. Before
presenting the proof, the advantage of alternating orders may be illustrated
by revisiting the linear example of \cite{Barzilai1988} with $A=diag(20,10,$
$2,1)$, $b=(1,1,1,1)^{\intercal}$, starting point $x_{0}=(0,0,0,0)^{\intercal
}$ and solution $x^{\ast}=(20^{-1},10^{-1},2^{-1},1)^{\intercal}$. The
convergence of ACX$^{3,2}$ will be compared with that of ACX$^{2}$, a purely
squared scheme equivalent to the CBB method in the linear case. The stopping
criterion is $\left\Vert \Delta x_{k}\right\Vert _{2}\leq10^{-8}$. Figure
\ref{fig:linear} shows the trajectory of $|e_{1}|$ on the horizontal axis and
$|e_{2}|$ on the vertical axis for both algorithms. The starting point is
$(20^{-1},10^{-1})$ at the top right. The ACX$^{2}$ method needs 34 gradient
evaluations before convergence but the ACX$^{3,2}$ only needs 20 (meanwhile,
the BB method requires 25 gradient evaluations and steepest descent needs
314). On the ACX$^{2}$ trajectory shown in dotted lines, each iteration
accomplishes a mild reduction of the errors. Substantial decrease of $|e_{2}|$
only occurs at iteration 5. The ACX$^{3,2}$ trajectory in dash-dotted lines
shows errors being annihilated more aggressively; $|e_{1}|$ and $|e_{2}|$ must
only be reduced substantially twice before convergence. While the benefits of
alternating are not always so large for small linear examples, Section
\ref{sec:Applications} will show how significant they can be in nonlinear
multivariate contexts.%

\begin{figure}[tbp]%
\begin{minipage}{0.48\textwidth}%
\begin{center}
\includegraphics[
height=7.0973cm,
width=7.0973cm
]%
{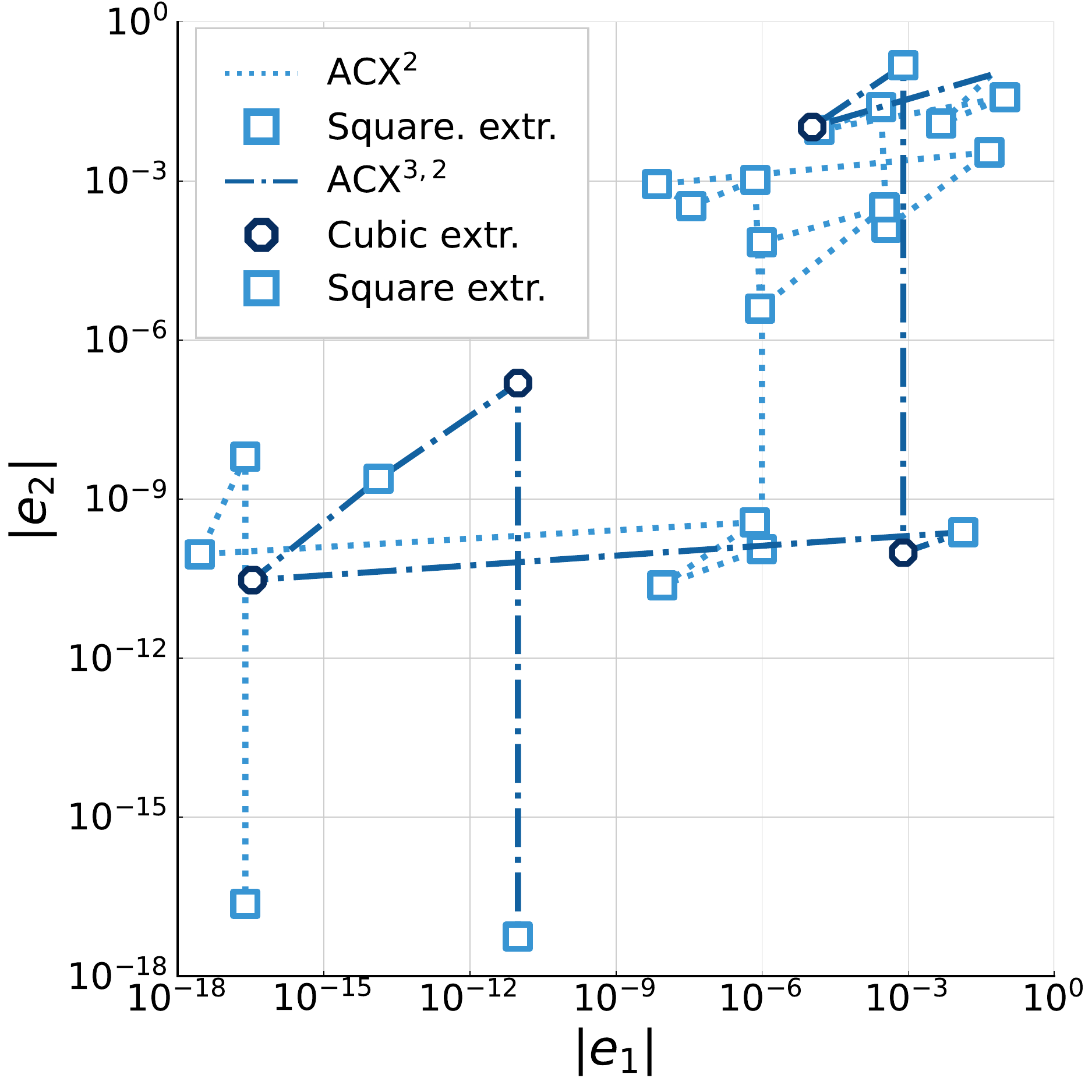}%
\end{center}
\caption
{Convergence of $|e_1|$ and $|e_2|$ for the linear system, from initial values of $(1/20,1/10)$ towards the solution $(0,0)$ for ACX$^2$ and ACX$^{3,2}%
$}%
\label{fig:linear}%
\end{minipage}%
\hfill%
\begin{minipage}{0.48\textwidth}%
\begin{center}
\includegraphics[
height=7.0973cm,
width=7.0973cm
]%
{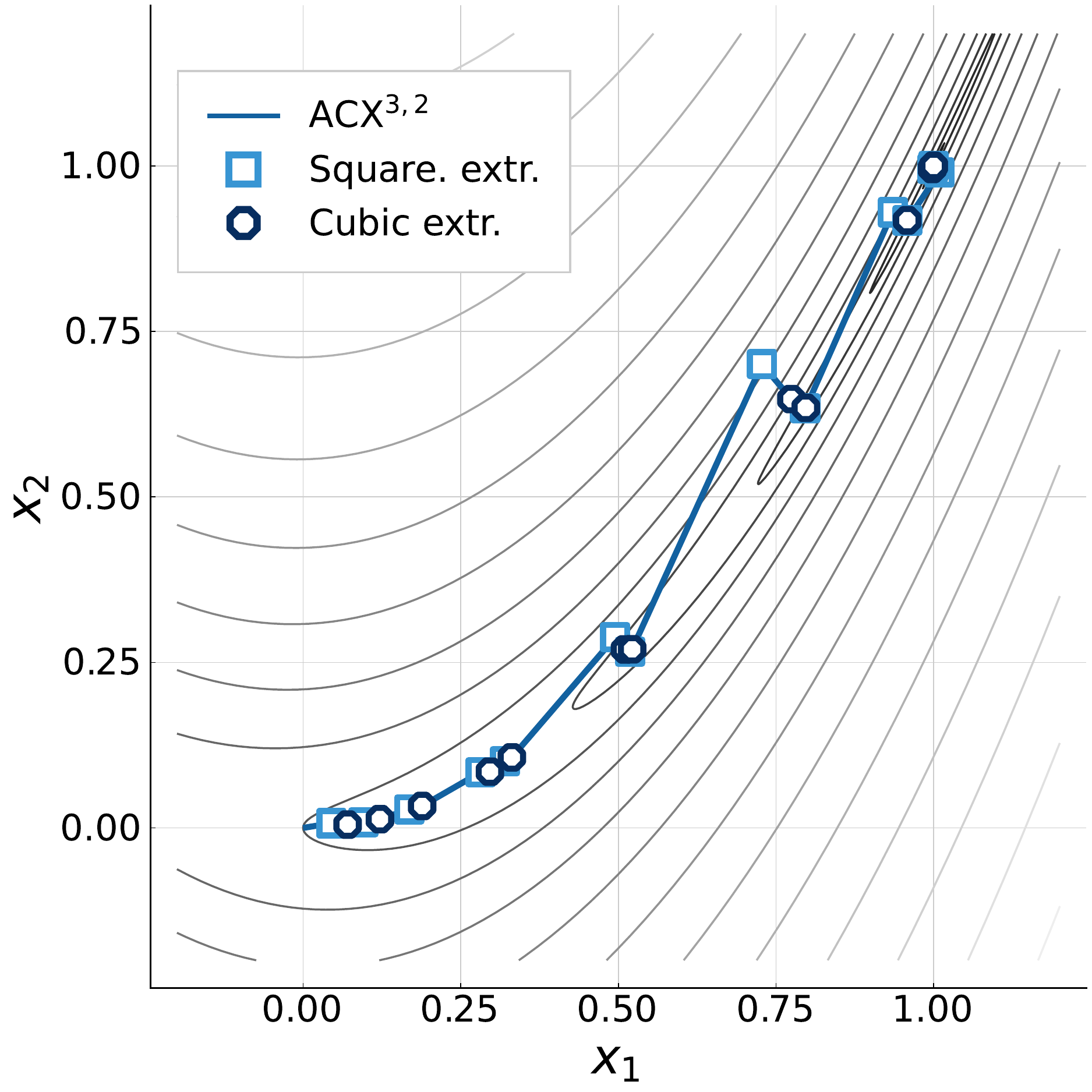}%
\end{center}
\caption{Convergence of ACX$^{3,2}%
$ for the 2-parameter Rosenbrock function from (0,0) towards the solution (1,1)\newline
}%
\label{fig:rosenbrock}%
\end{minipage}%
\end{figure}%

\section{Convergence of ACX%
\label{sec:proof}%
}

Studying the convergence of ACX systematically for linear systems of equations
is a good approximation for the behavior of a general mapping $F$ around its
fixed point $x^{\ast}$. \cite{Raydan2002} has shown that the CBB method
converges Q-linearly in an appropriate norm. The following proof extends the
result to any ACX algorithm.

Consider the linear system of equations defined $Qx=b$ where $Q$ is symmetric
positive definite. Define the elliptic norm%
\[
\left\Vert x\right\Vert _{Q^{-1}}=\sqrt{x^{\intercal}Q^{-1}x}%
\]
induced by the inner product $\left\langle .,.\right\rangle _{Q^{-1}}$%

\[
\left\langle x,y\right\rangle _{Q^{-1}}=x^{\intercal}Q^{-1}y\text{.}%
\]

The inner product satisfies the Cauchy-Schwarz inequality:%

\[
\left\Vert x\right\Vert _{Q^{-1}}^{2}\cdot\left\Vert y\right\Vert _{Q^{-1}%
}^{2}\geq\left\langle x,y\right\rangle _{Q^{-1}}^{2}%
\]
for vectors $x,y$.

To simplify the following computation, let us introduce the function%
\[
q_{p}(e)=e^{\intercal}Q^{p}e\text{.}%
\]
for some vector $e$ and the special case $q_{p}(e_{k})\equiv q_{p}$ for
$e_{k}$. Since $Q$ is positive definite, $q_{p}(e)\geq0$ $\forall$ $p\in%
\mathbb{R}
$. Observe notably that $\left\Vert e\right\Vert _{Q^{-1}}^{2}=e^{\intercal
}Q^{-1}e=q_{-1}(e)$.

Finally, let us introduce a useful lemma proven in appendix:

\begin{lemma}%
\label{lemma1}%
\label{Lemma1}For $x,y$, elements of a commutative ring, and $p\in%
\mathbb{N}
^{+}\setminus\{0,1\}$, $\left(  x+y\right)  ^{p}$ may be decomposed as%
\begin{equation}
(x+y)^{p}=x^{p}+y^{p}-\sum_{i=1}^{\left\lfloor p/2\right\rfloor }%
p(p-i-1)!\frac{(-xy)^{i}}{i!}\frac{(x+y)^{p-2j}}{(p-2j)!}\text{.}%
\label{eq:Lemma1}%
\end{equation}
where $\left\lfloor a\right\rfloor $ outputs the greatest integer less than or
equal to $a$.
\end{lemma}

\begin{theorem}
The sequence $\{x_{k}\}$ generated by the ACX$^{p_{1},...,p_{P}}$ method
(\ref{eq:ACX}) applied to the mapping $F(x)=x-(Qx-b)$ converges Q-linearly in
the norm $Q^{-1}$ for any set of values $p_{1},...,p_{P}$ with $p_{j}\geq2$.%
\label{theorem1}%

\end{theorem}

\begin{proof}
After a $p$-order extrapolation, the error $e_{k+1}$ may be expressed as%
\[
e_{k+1}=(I-\sigma_{k}^{(p)}Q)^{p}e_{k}\text{,}%
\]
where $p\geq2$ is the extrapolation order at iteration $k$ and $\sigma
_{k}^{(p)}=\frac{|-e_{k}^{\intercal}Q^{2p-1}e_{k}|}{e_{k}^{\intercal}%
Q^{2p}e_{k}}=\frac{q_{2p-1}}{q_{2p}}$. The squared $Q^{-1}$ norm of $e_{k+1}$
is%
\[
\left\Vert e_{k+1}\right\Vert _{Q^{-1}}^{2}=\left\Vert (I-\sigma_{k}%
^{(p)}Q)^{p}e_{k}\right\Vert _{Q^{-1}}^{2}=e_{k}^{\intercal}(I-\sigma
_{k}^{(p)}Q)^{p}Q^{-1}(I-\sigma_{k}^{(p)}Q)^{p}e_{k}\text{.}%
\]
Rearrange the terms on the right hand side:%
\[
\left\Vert e_{k+1}\right\Vert _{Q^{-1}}^{2}=e_{k}^{\intercal}Q^{-1}%
(I-\sigma_{k}^{(p)}Q)(I-\sigma_{k}^{(p)}Q)^{2p-1}e_{k}\text{.}%
\]
Rewrite the last parenthesis of the right hand side using (\ref{lemma1}) to
get%
\[
\left\Vert e_{k+1}\right\Vert _{Q^{-1}}^{2}=e_{k}^{\intercal}Q^{-1}%
(I-\sigma_{k}^{(p)}Q)\left[
\begin{array}
[c]{c}%
I+(-\sigma_{k}^{(p)}Q)^{2p-1}-\\
\sum_{i=1}^{\left\lfloor (2p-1)/2\right\rfloor }c(p,i)(\sigma_{k}^{(p)}%
Q)^{i}(I-\sigma_{k}^{(p)}Q)^{2p-1-2j}%
\end{array}
\right]  e_{k}\text{.}%
\]

where $c(p,i)=\frac{(2p-1)(2p-i-2)!}{i!(2p-1-2j)!}$. Note that $p\in%
\mathbb{N}
^{+}\setminus\{0,1\}\rightarrow\left\lfloor (2p-1)/2\right\rfloor =p-1$.
Multiplying back the terms outside the bracket and simplifying using the
$q_{p}(e)$ notation, we get%
\[
q_{-1}(e_{k+1})=%
\begin{array}
[c]{c}%
q_{-1}-\sigma_{k}^{(p)}q_{0}-(\sigma_{k}^{(p)})^{2p-1}(q_{2p-2}-\sigma
_{k}^{(p)}q_{2p-1})\\
-\sum_{i=1}^{p-1}c(p,i)(\sigma_{k}^{(p)})^{i}q_{-1}(Q^{i/2}(I-\sigma_{k}%
^{(p)}Q)^{p-i}e_{k})\text{.}%
\end{array}
\]

where $q_{-1}(Q^{i/2}(I-\sigma_{k}^{(p)}Q)^{p-i}e_{k})=e_{k}^{\intercal
}(I-\sigma_{k}^{(p)}Q)^{p-i}Q^{i-1}(I-\sigma_{k}^{(p)}Q)^{p-i}e_{k}$.
Factoring the first $q_{-1}$ of the right hand side, rewrite the expression as%
\[
q_{-1}(e_{k+1})=q_{-1}\cdot\left(  1-\theta_{1}-\theta_{2}-\theta_{3}\right)
\]
where

\begin{itemize}
\item $\theta_{1}=\sigma_{k}^{(p)}\frac{q_{0}}{q_{-1}}$

\item $\theta_{2}=\frac{(\sigma_{k}^{(p)})^{2p-1}}{q_{-1}}(q_{2p-2}-\sigma
_{k}^{(p)}q_{2p-1})=\frac{(\sigma_{k}^{(p)})^{2p-1}}{q_{-1}q_{2p}}%
(q_{2p-2}q_{2p}-q_{2p-1}^{2})$

\item $\theta_{3}=\sum_{i=1}^{p-1}c(p,i)(\sigma_{k}^{(p)})^{i}\frac
{q_{-1}(Q^{i/2}(I-\sigma_{k}^{(p)}Q)^{p-i}e_{k})}{q_{-1}(e_{k})}$.
\end{itemize}

Let us now show that $\theta_{1}\geq\frac{\lambda_{\min}}{\lambda_{\max}}$,
$\theta_{2}\geq0$, and $\theta_{3}\geq0$. Note that $\theta_{1}$ may be
written as the ratio of two Rayleigh quotients, with $\sigma_{k}^{(p)}%
=\frac{q_{2p-1}}{q_{2p}}=\left(  \frac{q_{2p}}{q_{2p-1}}\right)  ^{-1}=\left(
\frac{e_{k}^{\intercal}Q^{p-1/2}QQ^{p-1/2}e_{k}}{e_{k}^{\intercal}%
Q^{p-1/2}Q^{p-1/2}e_{k}}\right)  ^{-1}\in\lbrack\lambda_{\max}^{-1}%
,\lambda_{\min}^{-1}]$ and $\frac{q_{0}}{q_{-1}}=\frac{e_{k}^{\intercal
}Q^{-1/2}QQ^{-1/2}e_{k}}{e_{k}^{\intercal}Q^{-1/2}Q^{-1/2}e_{k}}\in
\lbrack\lambda_{\min},\lambda_{\max}]$. Hence, $\theta_{1}=\frac{q_{2p-1}%
}{q_{2p}}\frac{q_{0}}{q_{-1}}\in\left[  \frac{\lambda_{\min}}{\lambda_{\max}%
},\frac{\lambda_{\max}}{\lambda_{\min}}\right]  $.

Further, note that
\begin{align*}
q_{2p-2} &  =e_{k}^{\intercal}Q^{2p-2}e_{k}=\left\Vert Q^{p-1/2}%
e_{k}\right\Vert _{Q^{-1}}^{2}\\
q_{2p} &  =e_{k}^{\intercal}Q^{2p}e_{k}=\left\Vert Q^{p+1/2}e_{k}\right\Vert
_{Q^{-1}}^{2}\\
q_{2p-1} &  =e_{k}^{\intercal}Q^{2p-1}e_{k}=\langle Q^{p+1/2}e_{k}%
,Q^{p-1/2}e_{k}\rangle_{Q^{-1}}\text{.}%
\end{align*}

By the Cauchy-Schwarz inequality,%
\[
q_{2p-2}q_{2p}-q_{2p-1}^{2}=\left\Vert Q^{p-1/2}e_{k}\right\Vert _{Q^{-1}}%
^{2}\cdot\left\Vert Q^{p+1/2}e_{k}\right\Vert _{Q^{-1}}^{2}-\langle
Q^{p+1/2}e_{k},Q^{p-1/2}e_{k}\rangle_{Q^{-1}}^{2}\geq0\text{.}%
\]
Hence,%
\[
\theta_{2}=\frac{(\sigma_{k}^{(p)})^{2p-1}}{q_{-1}q_{2p}}(q_{2p-2}%
q_{2p}-q_{2p-1}^{2})\geq0\text{.}%
\]
Finally, for $\theta_{3}$, note that $c(p,i)>0$ for $p\geq2,i\leq p-1$ and
that all $\frac{q_{-1}(Q^{i/2}(I-\sigma_{k}^{(p)}Q)^{p-i}e_{k})}{q_{-1}%
(e_{k})}$ terms are Rayleigh quotients with minimum of zero when $\sigma
_{k}^{(p)}=\frac{1}{\lambda_{r}}$, $r\in\{1,...,n\}$.

Thus, we have%
\[
\frac{\left\Vert e_{k+1}\right\Vert _{Q^{-1}}}{\left\Vert e_{k}\right\Vert
_{Q^{-1}}}=\sqrt{1-\theta_{1}-\theta_{2}-\theta_{3}}\leq\sqrt{\frac
{\lambda_{\max}-\lambda_{\min}}{\lambda_{\max}}}%
\]
which establishes the result.
\end{proof}

\bigskip

We may convince ourselves that this linear convergence result is
representative of any nonlinear mapping $F$ in a neighborhood close to its
fixed point $x^{\ast}$. Taking a first-order Taylor approximation of $F$
applied $p$ times:
\[
F^{p}(x)=x^{\ast}+J^{p}\cdot(x-x^{\ast})+o(x-x^{\ast})
\]
where $J$ is the Jacobian of $F$ at $x^{\ast}$. For a $p$-order difference, we
have%
\[
\Delta^{p}x=(J-I)^{p}e+o(e)\text{.}%
\]
Applying a $p$-order extrapolation from $x_{k}$ (near $x^{\ast}$), the error
at iteration $k+1$ may be expressed as%
\[
e_{k+1}=(I-\sigma_{k}^{(p)}(I-J))^{p}e_{k}+o(e_{k})\text{.}%
\]

Hence, in a small neighborhood of $x^{\ast}$, ACX$^{3,2}$ and ACX$^{3,3,2}$
for nonlinear mappings should exhibit similar convergence as a linear map.

Theorem \ref{theorem1} also shows that from any starting point $x_{k}$ in the
neighborhood of $x^{\ast}$, all $p$-order extrapolation show the same worst
case scenario of $\sqrt{1-\lambda_{\min}/\lambda_{\max}}$. What differentiates
ACX is its ability to make better scenarios more likely.

Away from $x^{\ast}$, $J(x)$ may change substantially between each mapping,
and even more between extrapolations. Convergence is not guaranteed, but
alternating between squared and cubic extrapolations may be advantageous like
hybrid optimization algorithms are. Switching among constituent algorithms
helps escape situations such as zigzagging or locally flat objective functions
where a single algorithm would struggle. In fact, \cite{Roland2004} did
mention that SQUAREM -- a purely squared extrapolation scheme -- sometimes
experiences near stagnation or breakdown when $\Delta x$ and $\Delta^{2}x$ are
nearly orthogonal and $s^{BB1}$ or $s^{BB2}$ are used. For ACX$^{3,2}$ and
ACX$^{3,3,2}$, the probability that $\Delta x$ and $\Delta^{2}x$ be orthogonal
and that $\Delta^{2}x$ and $\Delta^{3}x$ also be orthogonal is lower. The next
section discusses implementation strategies to strike a good balance between
speed and stability away from $x^{\ast}$.

\section{Implementation detail and stability%
\label{sec:Implementation}%
}

This section discusses implementation detail of ACX in various situations. The
proposed set of parameter values show good empirical properties but others
could probably work well too.

\subsection*{Adaptive step size for gradient descent acceleration}

Consider minimizing the function $f$ by gradient descent. The mapping is%
\[
F(x)=x-\alpha\nabla f(x)\text{.}%
\]

To be amenable to ACX acceleration, the descent step size $\alpha$ must be
constant within each extrapolation cycle. It must also remain within an
acceptable range of values since ACX is based on second- and third-order
differences. An excessively small $\alpha$ may lead to small $\Delta^{2}x$ or
$\Delta^{3}x$ and large $\sigma^{(2)}$ or $\sigma^{(3)}$, resulting in
imprecise extrapolations when $f$ has weak curvature. An excessively large
$\alpha$ could on the other hand lead to zigzagging and at worst divergence of
the algorithm.

A simple adaptive procedure for $\alpha$ can improve the chance $\sigma^{(p)}$
remains within reasonable bounds as often as possible. The initial step size
is the largest possible $\alpha_{0}$ which \textit{i}) satisfies the
Armijo--Goldstein condition (\cite{Armijo1966}): $f(x_{0}-\alpha_{0}\nabla
f(x_{0}))\leq f(x_{0})-c_{AG}\alpha_{0}\left\Vert \nabla f(x_{0})\right\Vert $
where $c_{AG}\in(0,1)$ is some constant\ and \textit{ii}) does not result in a
large increase in the gradient norm: $\left\Vert \triangledown f(x_{0}%
-\alpha_{0}\nabla f(x_{0}))\right\Vert _{2}\leq L_{n}\left\Vert \nabla
f(x_{0})\right\Vert _{2}$ where $L_{n}\geq1$ is some
constant.\footnote{Additionally, to stabilize the start of ACX$^{3,2}$ and
ACX$^{3,3,2}$, $\sigma_{0}^{(2)}$ is computed at the first iteration. If it is
below 1, a sign $\alpha$ is rather large, the algorithm starts with a squared
extrapolation (an interpolation in that case) rather than a cubic one.} After
each iteration $k$, set $\alpha_{k+1}=\alpha_{k}\theta$ if $\sigma_{k}%
^{(p)}<\underline{L}_{\sigma}$ and $\alpha_{k+1}=\alpha_{k}/\theta$ if
$\sigma_{k}^{(p)}>\bar{L}_{\sigma}$ where $\theta>1$ and where $\underline{L}%
_{\sigma}$ and $\bar{L}_{\sigma}>\underline{L}_{\sigma}$ are lower and upper
thresholds within which $\sigma_{k}^{(p)}$ should preferably remain. The
rationale for this simple step function is that $\sigma_{k}^{(p)}$ sometimes
take very large or very small values for a single iteration. Overcorrecting
$\alpha_{k+1}$ based on such iteration would be a disproportionate response
and lead to worse results in the following iterations. In the empirical
section, the parameters are set to $c_{AG}=0.25$, $L_{n}=2$, $\theta=1.5$,
$\underline{L}_{\sigma}=1$ and $\bar{L}_{\sigma}=2$.\footnote{Note that near
the optimum, if $\alpha_{k}$ is small, $\Delta^{p}x_{k}$ may be too small for
machine precision and lead to an imprecise $\sigma_{k}^{(p)}$. This is
prevented with a progressive approach. Whenever $||\Delta^{p}x_{k}||_{\infty}$
is small considering the available machine precision ($||\Delta^{p}%
x_{k}||_{\infty}<10^{-50}$ in the applications), $\sigma_{k}^{(p)}$ is set to
$1$ and $\alpha_{k+1}$ is set to $\min(1,2^{1+t}\alpha_{k})$, where $t$ is the
total number of times the same situation has occurred in the past.}

\subsection*{Accelerating mappings other than gradient descent}

\subsubsection*{Constraining $\sigma^{(p)}$%
\label{sec:sig_constr}%
}

For mappings with guaranteed improvement in the objective such as the EM
algorithm, the step length may be constrained to $\underline{\sigma}%
^{(p)}=\max(1,\sigma^{(p)})$. Otherwise, in some scenarios, $\left\vert
\langle\Delta^{p}x,\Delta^{p-1}x\rangle\right\vert $ could be close to zero,
making $\sigma^{(p)}$ small as well and leading to slow progress. If the
underlying mapping has guaranteed progress, then $\underline{\sigma}^{(p)}=1$
ensures that ACX makes the same progress as the last mapping.\footnote{This
was also suggested by \cite{Roland2008}. Note that it would not be appropriate
for applications where we often have $\sigma^{p}\in(0,1)$, such as the mapping
operators considered by \cite{Lemarechal1971} with Lipschitz constants
$L\in(0,1)$. To limit zigzagging, a solution is replacing $F$ by $F\circ F$.}

\subsubsection*{Stabilization mappings%
\label{sec:Stab}%
}

In many mapping applications such as the EM algorithm or the
Majorize-Minimization (MM) algorithm (\cite{Lange2016}), the mapping
$F(x_{k})$ takes the form of a constrained optimization given the parameter
values of the previous iteration. For example, for the MM algorithm, the
objective function $f(x)$ is maximized iteratively by the constrained
maximization of a surrogate function $g(x)$. The mapping is $F(x_{k})=\arg
\max_{x}g(x|x_{k-1})$. Since the starting point $x_{0}$ is not in general a
constrained maximum, the value of the objective function $f$ can improve
significantly following the first iteration. In the subsequent iterations,
progress is typically much slower as $x_{k}$ steadily converges from
constrained maximum to constrained maximum toward its fixed point. This also
means that the change in $x\ $after the first mapping may be sizable but be
comparatively modest in the following iterations. Since ACX relies on
extrapolation, using this initial mapping may provide little information on
the true direction of the fixed point. To improve the accuracy of the
extrapolation, an initial \textquotedblleft stabilization
mapping\textquotedblright\ may be computed before each extrapolation. In
practice, it is difficult to anticipate whether convergence will improve
sufficiently to warrant the investment in this extra mapping. In the
applications, it was beneficial for the Poisson mixture application, for
ACX$^{2}$ specifically, and not beneficial for the other applications.

\subsection*{Non-monotonicity%
\label{sec:Global}%
}

ACX does not guarantee steady improvement in the objective at every iteration.
Monotonicity could be enforced for ACX by either reducing the step size or
simply falling back on a recent iterate with guaranteed improvement. This
however may not be beneficial. Consider Figure \ref{fig:rosenbrock} showing
the convergence of ACX$^{3,2}$ for the two-variable Rosenbrock function
starting from $(0,0)$. As can be seen from the contour lines, some of the most
fruitful steps toward the minimum $(1,1)$ are also causing temporary setbacks
in the value of the objective. But the lost progress is quickly regained a few
steps later. Ample testing has shown that reducing the size of these steps to
enforce monotonicity often slows down convergence, sometimes dramatically.
Since a fast algorithm is more valuable than a slow monotonic one, it was not implemented.

\subsection{Backtracking%
\label{sec:Backtracking}%
}

In rare occasions, the algorithm may reach parameter values where the gradient
is undefined or infinite. Such outcome may be the result of $\sigma^{(p)}$
being too large or, for gradient descent, $\alpha$ being too large. In these
situations, the following procedure is proposed.

The algorithm resumes at the best past iterate, labeled $k_{B}$, either the
one with the smallest norm ($\left\Vert \nabla f(x_{k_{B}})\right\Vert $ or
$\left\Vert \Delta x_{k_{B}}\right\Vert $), or with the smallest objective
value $f(x_{k_{B}})$. To increase the stability of the extrapolation, replace
$\sigma_{k_{B}+i}^{(p)}\rightarrow\rho_{\sigma}\sigma_{k_{B}+i}^{(p)}$ for
$i=1,2...$ with $0<\rho_{\sigma}<1$. For gradient descent acceleration, also
replace $\alpha_{k_{B}+1}\rightarrow\rho_{\alpha}\alpha_{k_{B}+1}$ where
$0<\rho_{\alpha}<1$, (and adjust $\alpha$ as usual in the following
iterations). If the norm or the objective eventually improves, the algorithm
resumes with normal extrapolation step lengths. If a new infeasible iterate
occurs before an improvement, the algorithm backtracks again to $k_{B}$ and
restarts a second time with reductions $(\rho_{\alpha}^{2},\rho_{\sigma}^{2}%
)$. Then, in case of failure, with $(\rho_{\alpha}^{3},\rho_{\sigma}^{3})$,
$...$. Such successive reductions in $\sigma^{(p)}$ and $\alpha$ should
eventually lead to progress in a similar way to regular gradient descent with
backtracking and make algorithm failure unlikely. In the empirical section,
$\rho_{\alpha}=\rho_{\sigma}=2$.

\subsection{Bounds checking%
\label{sec:Boundary}%
}

Stalling may occur if an extrapolation leads to a saddle point or to a portion
of the parameter space where $F(x)$ is defined but $\Delta x$ is very small.
Bound checks may prevent this situation If such problematic parameter region
is known to the users. Let $S\in%
\mathbb{R}
^{n}$ be the set of feasible starting points that may be represented as the
Cartesian product of $n$ open intervals: $S=\prod_{i=1}^{n}I^{(i)}$ where
$I^{(i)}=(x_{\min}^{(i)},x_{\max}^{(i)})$ $i\in(1,...,n)$. This simple
representation -- sometimes referred to as a box constraint -- is appropriate
in many applications but could potentially be generalized. Let $x_{k}\in S$ be
a starting point and $x_{k+1}$ be the next extrapolation. The next iterate
with bounds checks is%
\[
\bar{x}_{k+1}^{(i)}=\max(\min(x_{k+1}^{(i)},\omega x_{\max}^{(i)}%
+(1-\omega)x_{k}^{(i)}),\omega x_{\min}^{(i)}+(1-\omega)x_{k}^{(i)})\quad
i\in(1,...,n)
\]
where $\omega\in(0,1)$. This strategy keeps each $\bar{x}_{k+1}^{(i)}$ within
bounds and ensures that no extrapolation covers more than a fraction $\omega$
of the distance between $x_{\max}^{(i)}-x_{k}^{(i)}$ or $x_{k}^{(i)}-x_{\min
}^{(i)}$. A bound may still be reached asymptotically if it does contain the
fixed point $x^{\ast}$. Note that the same bounds checking may be performed
with gradient descent acceleration after each mapping to implement constrained optimization.

\section{Applications
\label{sec:Applications}%
}

This section compares ACX to fast alternatives. For gradient descent
acceleration, the minimization problems are a multivariate Rosenbrock function
with or without constraints,\ a logistic regression, and 96 unconstrained
problems from the CUTEst collection. For general mapping acceleration, the
applications are the EM algorithm for a Poisson admixture model, the EM
algorithm for a proportional hazard regression with interval censoring,
alternating least squares (ALS) applied to rank tensor decomposition, the
power method for finding dominant eigenvalues, and the method of alternating
projection (\cite{vonNeumann1950}, \cite{Halperin1962}) applied to regressions
with high-dimensional fixed effects.

Since the various algorithms vary greatly in terms of gradient and objective
evaluations as well as internal computation, their performances are assessed
by CPU time, presented via the performance profiles of \cite{Dolan2002}. These
graphs show how often each algorithm was within a certain multiple of the best
compute time. Appendix \ref{app:results} also shows the average number
objective function evaluations, gradient evaluations or mappings, evaluation
time (for the draws that converged), and convergence rates.

Most computations were performed in Julia\footnote{Its just-in-time compiler
guarantees little computing overhead. This is important to get accurate ideas
of the relative number of operations required for each method.} with the
exception of the tensor canonical decomposition performed in MATLAB and the
alternating projections application which compares packages of various
languages. The benchmark stopping criterion was $\left\Vert \triangledown
f(x)\right\Vert _{\infty}\leq10^{-7}$ for gradient descent acceleration and
$\left\Vert \Delta x\right\Vert _{\infty}\leq10^{-7}$ for general mapping
applications. To ensure meaningful comparisons, draws for which different
algorithms converged on divergent objective values were discarded. The precise
condition was%
\begin{equation}
|f(x_{T,i})-\min_{j}(f(x_{T,j}))|<10^{-5}\text{\quad}\forall i,\label{eq:crit}%
\end{equation}
where $x_{T,i}$ is the final iterate of any algorithm $i$, and $\min
_{j}(f(x_{T,j}))$ is the minimum over all final iterates $j$, including for
algorithms that have not converged. This rather stringent criterion was
introduced with the CUTEst problems in mind. For each application, 2000 draws
were computed, except for the CUTEst problems and the alternating projections
application. In all applications, the gradient norm was used to track the
progress of ACX. For each application, Appendix \ref{app:results} presents
additional implementation detail and statistics. Information on software,
package versions and hardware used is provided in Appendix \ref{app:Software}.

\subsection*{Gradient descent applications}

For gradient descent applications, the performances of ACX$^{2}$, ACX$^{3,2}$
and ACX$^{3,3,2}$ were compared with the L-BFGS algorithm and a N-CG method
from the package Optim.jl (\cite{mogensen2018optim}). The L-BFGS is
implemented with Hager-Zhang line search and a window of 10 past iterates to
build the Hessian approximation. The N-CG is also implemented with Hager-Zhang
line search. The algorithm combines features of \cite{Hager2006} and
\cite{Hager2013} and multiple revisions to the code since publication.

\subsubsection*{The Rosenbrock function}

The Rosenbrock is a well-known test bed for new algorithms. The test
specification involved finding the unconstrained minimum of a 1000-parameter
version of the function%

\[
f(x^{(1)},...,x^{(N)})=\sum_{i=1}^{N/2}\left[  100\left(  \left(
x^{(2i-1)}\right)  ^{2}-x^{(2i)}\right)  ^{2}+\left(  x^{(2i-1)}-1\right)
^{2}\right]  \text{\qquad}N=1000\text{.}%
\]
For the unconstrained minimization, the starting points $x_{0}^{(i)}$ were
drawn from uniform distributions $U[-5,5]$. A constrained minimization was
also implemented with upper bound $x_{\max}^{(i)}$ sampled from $U[0,1]$ and
starting points $x_{0}^{(i)}$ sampled from $U[-5,0]$. For L-BFGS and N-CG, the
constraint was implemented with barrier penalty using Optim.jl's
fminbox.\footnote{See
https://julianlsolvers.github.io/Optim.jl/stable/\#user/minimization/\#box-constrained-optimization
and
https://github.com/JuliaNLSolvers/Optim.jl/blob/adc5b277b3f915c25233b45f8f2dd61006815e63/src/multivariate/solvers/constrained/fminbox.jl
for more detail.}

The unconstrained minimization results are displayed in Figures
\ref{fig:Rosenbrock} as well as Appendix Table \ref{tab:Rosenbrock}. The three
ACX algorithms show advantageous performances in terms of gradient
evaluations, objective evaluations and compute time, with ACX$^{3,3,2}$ being
fastest 80\% of the time.

The constrained minimization results are shown in Figure
\ref{fig:RosenbrockConstr} and Appendix Table \ref{tab:RosenbrockConstr}.
Here, the simple box constraint implementation of Section \ref{sec:Boundary}
applied to gradient descent combined with ACX is surprisingly efficient. It is
30 to 40 times faster than L-BFGS or N-CG with barrier penalty. This
encouraging result suggests ACX could also be efficient with a more general
set of linear and non-linear constraints.%

\begin{figure}[tbp]%
\begin{minipage}{0.48\textwidth}%
\raisebox{-0cm}{\includegraphics[
height=7.0973cm,
width=7.0973cm
]%
{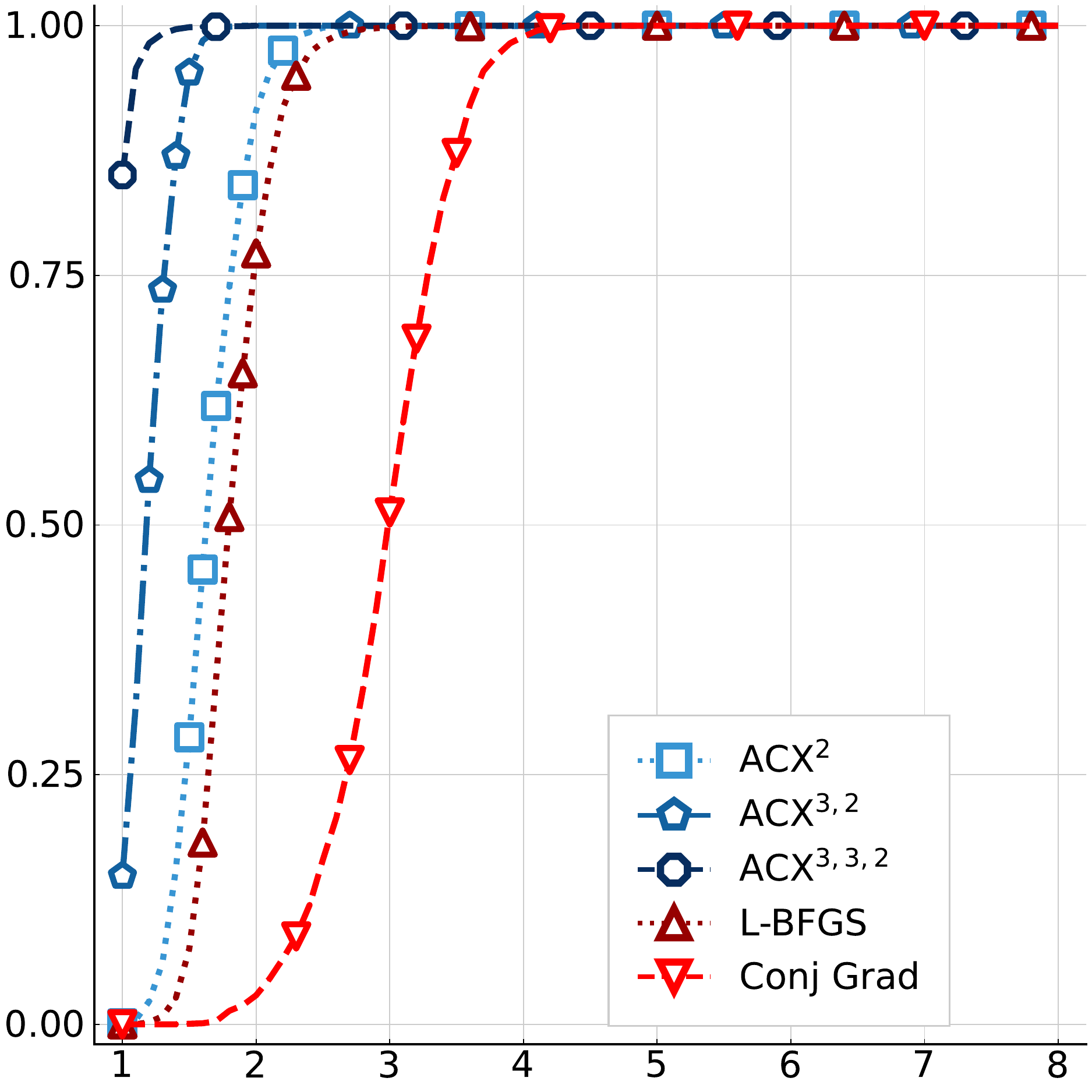}%
}
\caption{Performance profiles for the Unconstrained 1000-parameter Rosenbrock}%
\label{fig:Rosenbrock}%
\end{minipage}%
\hfill%
\begin{minipage}{0.48\textwidth}%
\raisebox{-0cm}{\includegraphics[
height=7.0973cm,
width=7.0973cm
]%
{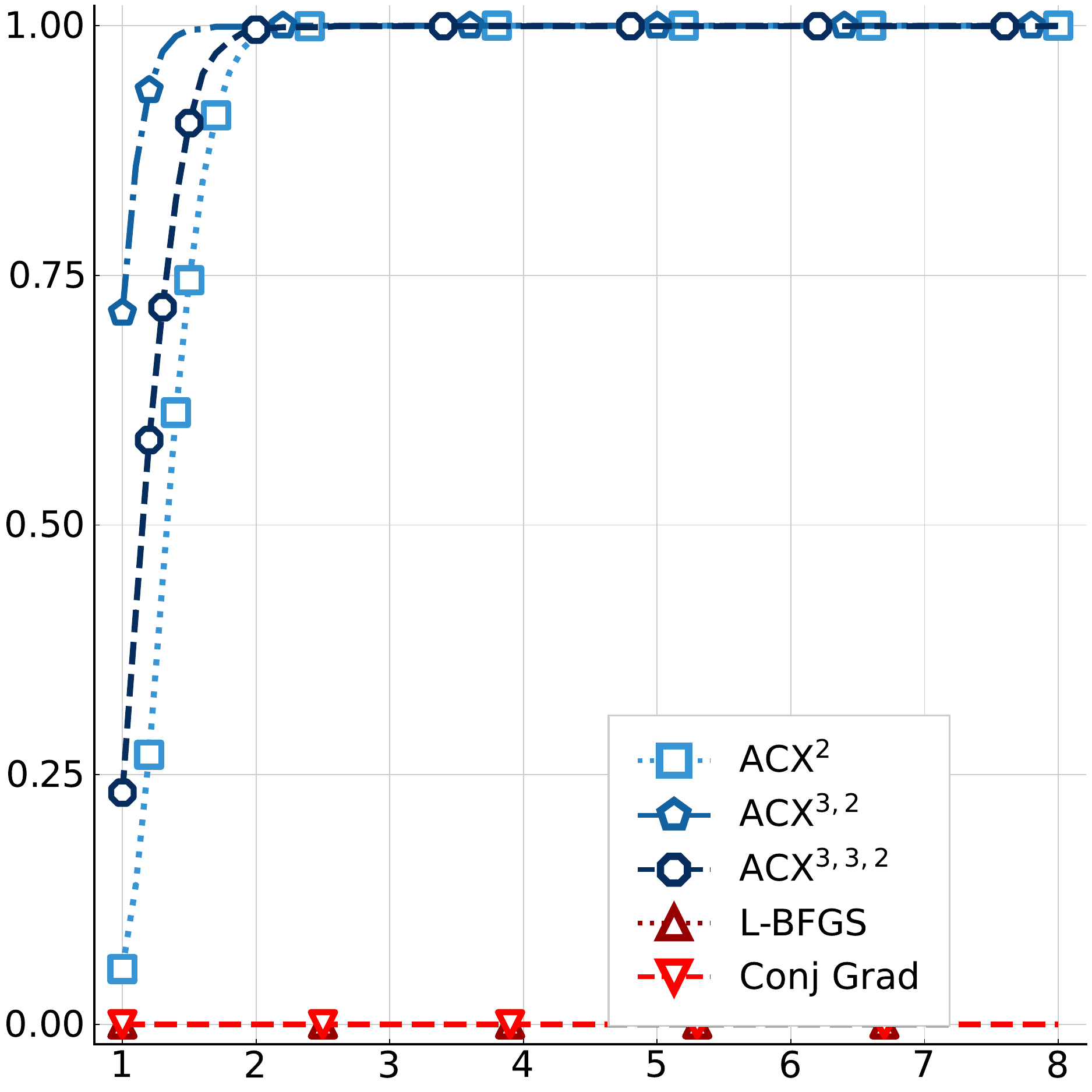}%
}
\caption{Performance profiles for the Constrained 1000-parameter Rosenbrock}%
\label{fig:RosenbrockConstr}%
\end{minipage}%
\end{figure}%

\subsubsection*{A logistic regression}

An advantage of ACX is its limited reliance on objective functions. This may
provide clear benefits for an application like the logistic regression for
which the log likelihood requires taking logs:%

\[
l(y|X,\beta)=\sum_{i=1}^{n}\left[  y_{i}\times x_{i}^{\intercal}\beta
-\log(1+\exp(x_{i}^{\intercal}\beta))\right]  \text{,}%
\]
where $y\in%
\mathbb{R}
^{n},X\in%
\mathbb{R}
^{n\times m}$, but the gradient does not.

To illustrate this advantage, a simulation was conducted with $n=2000$,
$m=100$, with coefficients and the covariates drawn from uniform distributions
$U[-1,1]$ and $X$ containing a column of ones. For each draw, the starting
point was $\beta_{0}=\mathbf{0}$. The obvious speed gains from ACX$^{3,2}$ and
ACX$^{3,3,2}$ can be seen in Figure \ref{fig:logistic}. Appendix Table
\ref{tab:logistic} shows ACX$^{3,2}$ and ACX$^{3,3,2}$ needed slightly fewer
gradient evaluations than the alternatives and close to no objective function
evaluations, making them 6 times faster than the L-BFGS and 10 times faster
than the N-CG.%

\begin{figure}[tbp]%
\begin{minipage}{0.48\textwidth}%
\raisebox{-0cm}{\includegraphics[
height=7.0951cm,
width=7.0951cm
]%
{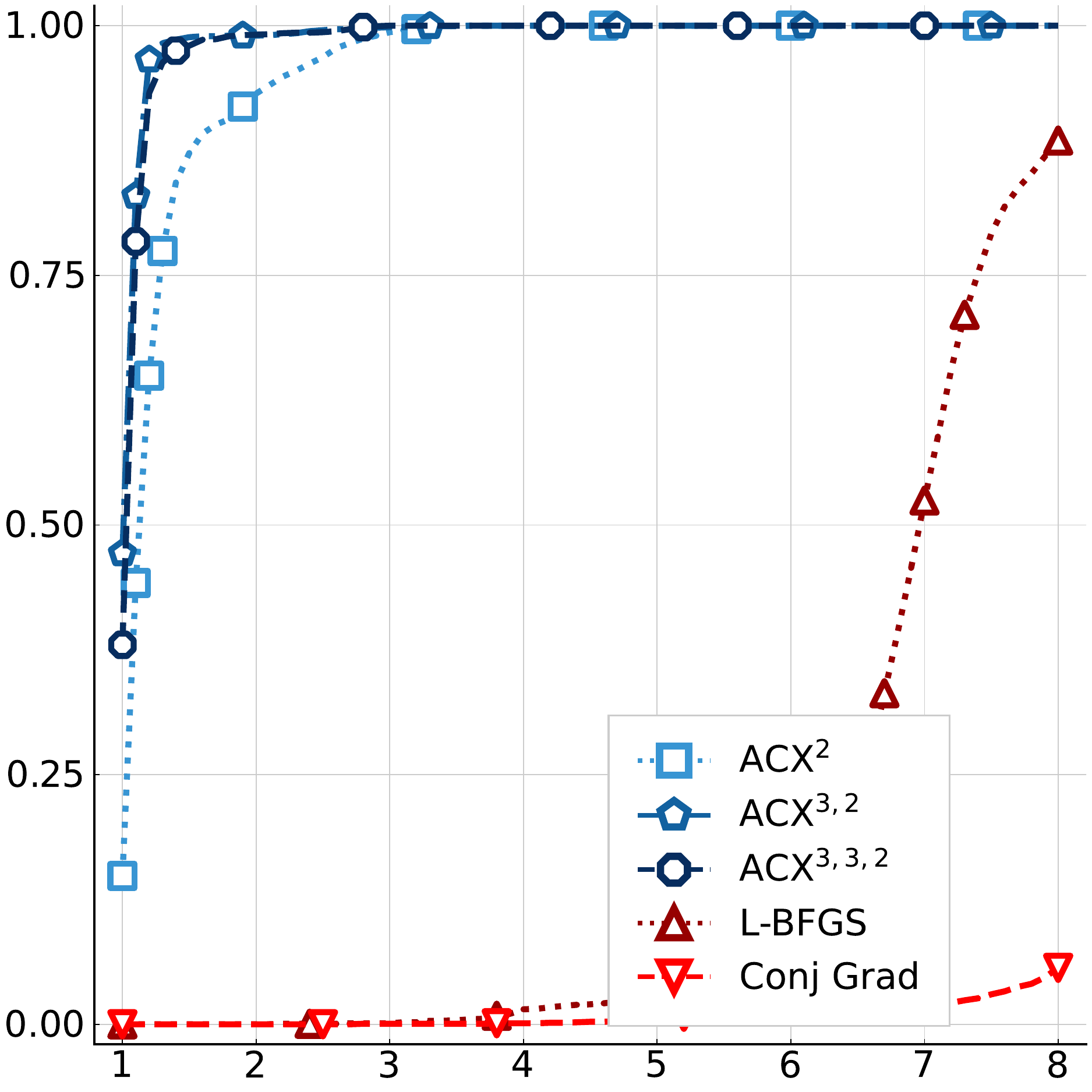}%
}
\caption{Performance profiles for the logistic regression\newline}%
\label{fig:logistic}%
\end{minipage}%
\hfill%
\begin{minipage}{0.48\textwidth}%
\raisebox{-0cm}{\includegraphics[
height=7.0973cm,
width=7.0973cm
]%
{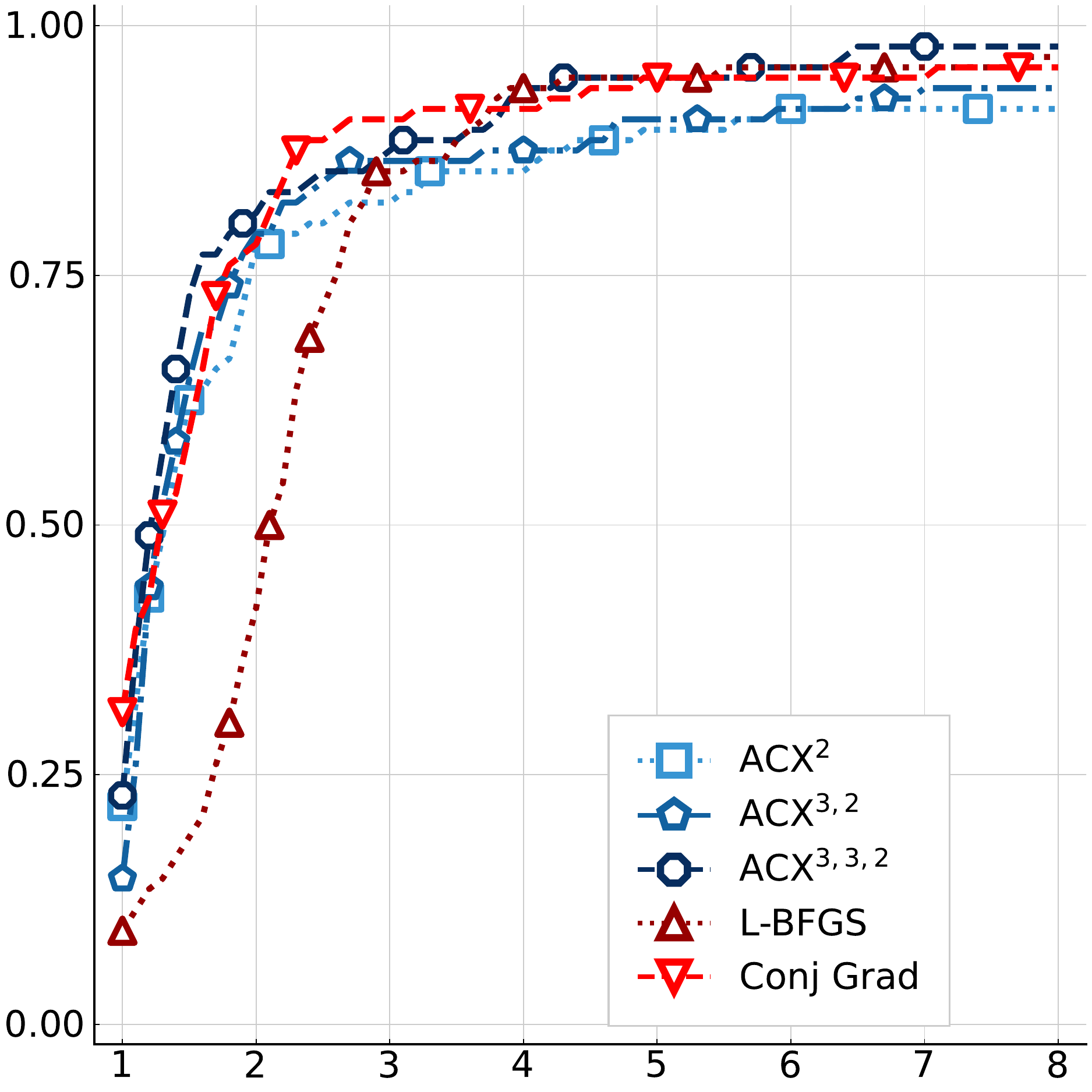}%
}
\caption
{Performance profiles for 96 unconstrained optimization from te CUTEst collection (each problem is one draw)}%
\label{fig:CUTEst}%
\end{minipage}%
\end{figure}%

\subsubsection*{The set of unconstrained minimization problems from the CUTEst
suite}

The CUTEst problem set -- successor of CUTE (Constrained and Unconstrained
Testing Environment) and CUTEr -- has become the benchmark for prototyping new
optimization algorithms. For testing, the unconstrained problems with
objective function defined as either quadratic, sum of squares or
\textquotedblleft other\textquotedblright\ were selected, numbering 154 in
January 2021. Of these were excluded those with a zero initial gradient or
with fewer than 50 parameters. When possible, the default number of parameters
was used. When it was below 50, the number was set to 50 or 100 if available.

Algorithms were allowed to run for at most 100 seconds. A problem was excluded
if no algorithm converged before 10 seconds or if all algorithms converged in
fewer than 5 mappings. To avoid stalling, the initial stopping criterion was
$\left\Vert \triangledown f(x)\right\Vert _{\infty}<10^{-5}$. When the final
objectives diverged such that criterion (\ref{eq:crit}) was not met, the
tolerance was progressively decreased, down to $\left\Vert \triangledown
f(x)\right\Vert _{\infty}<10^{-8}$. If the result still diverged, the number
of parameters was reduced. In the end, 96 problems were kept for comparison.

The resulting performance profiles displayed in Figure \ref{fig:CUTEst} show
ACX$^{3,3,2}$ outperforming the other ACX and frequently the other algorithms
as well. The N-CG also did very well while the L-BFGS was often slower. The
detail on the number of gradient and objective function evaluations, compute
time, convergence status and minimum objective attained is available at sites.google.com/site/nicolaslepagesaucier/output\_CUTEst.txt.

\subsection*{Various mapping applications}

\subsubsection*{The EM algorithm for Poisson admixture model%
\label{sec:mixtures}%
}

The EM algorithm is a ubiquitous method in statisticians' toolbox. While
stable, it can be notoriously slow to converge, motivating the development of
many acceleration methods. 

A classic implementation of the EM algorithm is \cite{Hasselblad1969} who
models of the death notices of women over 80 years old reported in the London
Times over a period of three years. Table \ref{tab:deaths} reproduces the data.%

\begin{table}[!htbp] \centering
\begin{minipage}{0.92\textwidth}%
\captionsetup{justification=raggedright,singlelinecheck=false}%
\caption{Number of death notices}\vspace{0.08cm}\label{TableKey}%
\begin{tabular}
[c]{ccccccccccc}\hline
&  &  &  &  &  &  &  &  &  & \vspace{-0.4cm}\\
\multicolumn{1}{l}{Observed death count ($i$)} & 0 & 1 & 2 & 3 & 4 & 5 & 6 &
7 & 8 & 9\\
\multicolumn{1}{l}{Frequency of occurrence ($y_{i}$)} & 162 & 267 & 271 &
185 & 111 & 61 & 27 & 8 & 3 & 1\vspace{0.1cm}\\\hline
\end{tabular}%
\label{tab:deaths}%
\end{minipage}%
\end{table}%

The data are modeled as a mixture of two Poisson distributions to capture
higher death rates during winter. The likelihood is%
\[
L(y^{(0)}...y^{(9)}|\mu^{(1)},\mu^{(2)},\pi)=\prod_{i=0}^{9}\left[  \pi
e^{-\mu^{(1)}}\frac{(\mu^{(1)})^{i}}{i!}+(1-\pi)e^{-\mu^{(2)}}\frac{(\mu
^{(2)})^{i}}{i!}\right]  ^{y^{(i)}}%
\]
where $\mu_{1}$ and $\mu_{2}$ are the means of the distributions of
subpopulations $1$ and $2$ and $\pi$ is the probability that a random
individual is part of subpopulation $1$. The EM algorithm map is%
\[
\mu_{k+1}^{(1)}=\frac{\sum_{i=0}^{9}y^{(i)}iw_{k}^{(i)}}{\sum_{i=0}^{9}%
y^{(i)}w_{k}^{(i)}}\text{; }\mu_{k+1}^{(2)}=\frac{\sum_{i=0}^{9}%
y^{(i)}i(1-w_{k}^{(i)})}{\sum_{i=0}^{9}y^{(i)}(1-w_{k}^{(i)})}\text{; }%
\pi_{k+1}=\frac{\sum_{i=0}^{9}y^{(i)}w_{k}^{(i)}}{\sum_{i=0}^{9}y^{(i)}}%
\]

where%
\[
w_{k}^{(i)}=\frac{\pi_{k}e^{-\mu_{k}^{(1)}}\left(  \mu_{k}^{(1)}\right)  ^{i}%
}{\pi_{k}e^{-\mu_{k}^{(1)}}\left(  \mu_{k}^{(1)}\right)  ^{i}+(1-\pi
_{k})e^{-\mu_{k}^{(2)}}\left(  \mu_{k}^{(2)}\right)  ^{i}}\text{.}%
\]

For the experiments, random starting points were sampled from uniform
distributions $\pi_{0}\backsim U[0.05,0.95]$, $\mu_{0}^{(i)}\backsim
U[0,20]$\quad$i=1,2$. ACX was implemented with bound checks $\mu^{(i)}\geq0$
$i=1,2$ and $\pi\in\lbrack0,1]$ with a buffer of $\omega=0.9$ and
stabilization mapping. Its performances were compared with those of QNAMM (3)
and DAAREM (2)\footnote{As recommended by the authors, we use $\min
(10,\left\lceil n/2\right\rceil )$ lags where $n$ is the number of
parameters.}, both adapted to Julia from their respective R packages.

The results are displayed in Figure \ref{fig:mixtures} and appendix Table
\ref{tab:mixtures}. With only three parameters to estimate, ACX$^{3,2}$
performed slightly better than ACX$^{3,3,2}$. Both outperformed QNAMM (3) and
DAAREM (2) on all metrics, probably a result of the need to check the
objective and their general computational burdens.%

\begin{figure}[tbp]%
\begin{minipage}{0.48\textwidth}%
%

\raisebox{-0cm}{\includegraphics[
height=7.0973cm,
width=7.0973cm
]%
{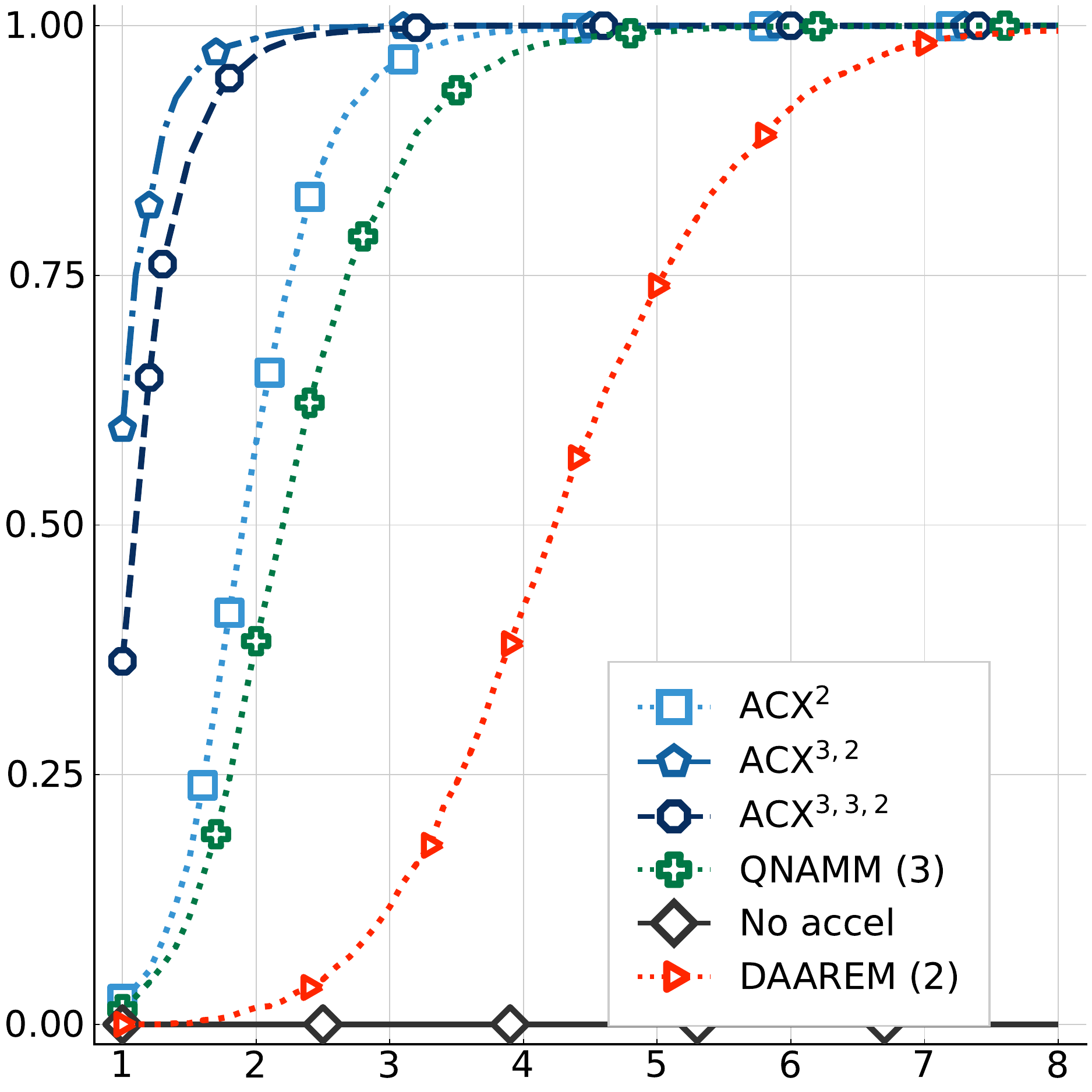}%
}
\caption
{Performance profiles for  the EM algorithm acceleration of Poisson admixtures}%
\label{fig:mixtures}%
\end{minipage}%
\hfill%
\begin{minipage}{0.48\textwidth}
\raisebox{-0cm}{\includegraphics[
height=7.0973cm,
width=7.0973cm
]%
{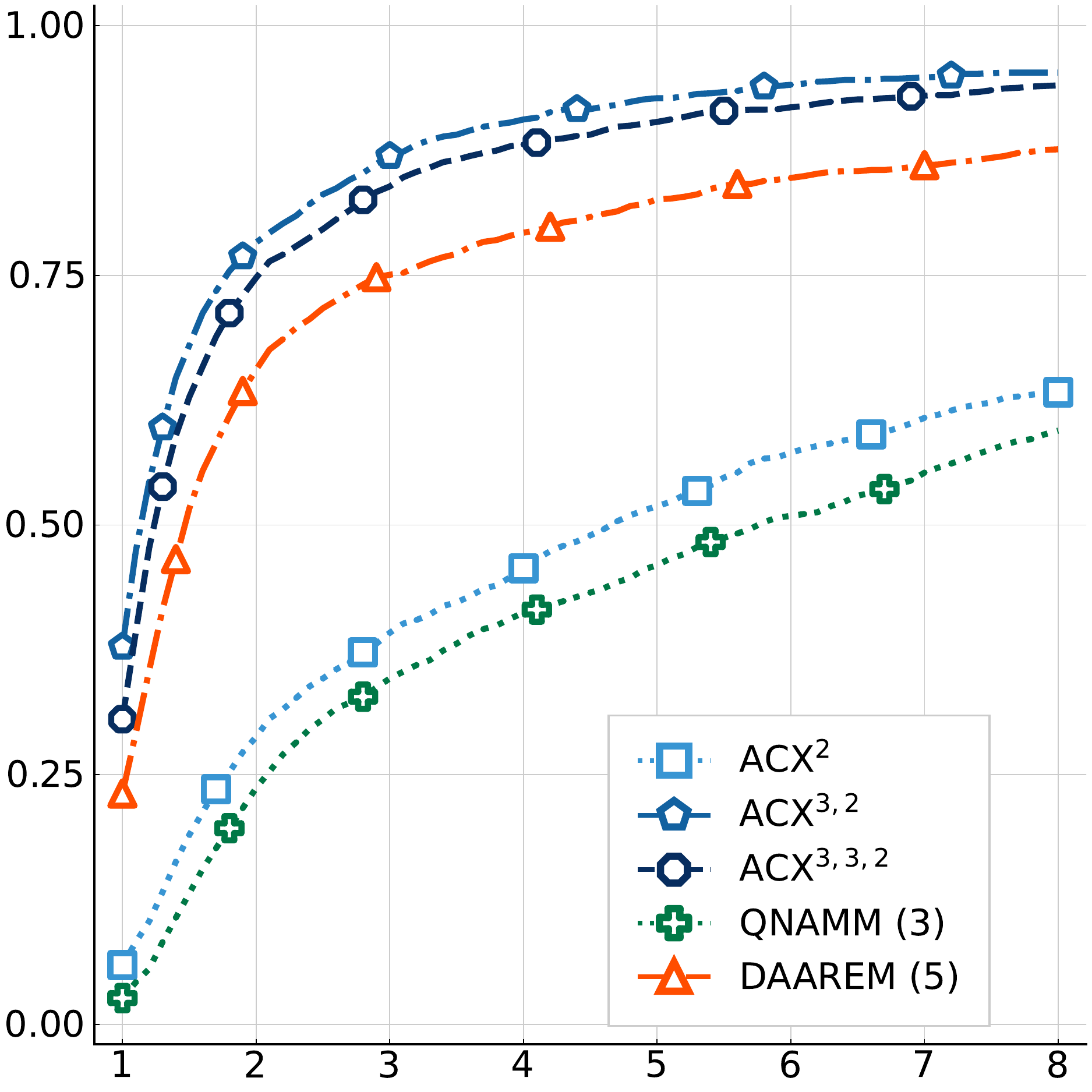}%
}
\caption
{Performance profiles for the EM algorithm for proportional hazard regression with interval censoring
}%
\label{fig:cens}%
\end{minipage}%
\end{figure}%

\subsubsection*{The EM algorithm for proportional hazard regression with
interval censoring}

Another common application of the EM algorithm is the mixed proportional
hazard model. \cite{Wang2016} proposed a new method for estimating a
semiparametric proportional hazard model with interval censoring, a common
complication arising in medical and social studies. Their EM estimation relies
on a two-stage data augmentation with latent Poisson random variables and a
monotone spline representation of the baseline hazard function. The algorithm
is light and simple to implement (see \cite{Wang2016} for details) yet may
benefit greatly from acceleration.

The simulation settings and codes of \cite{Henderson2019} were adapted to
Julia to allow for meaningful comparison. The likelihood for an individual
observation is
\[
L(\delta_{1},\delta_{2},\delta_{3},\mathbf{x})=F(R|\mathbf{x})^{\delta_{1}%
}\{F(R|\mathbf{x})-F(L|\mathbf{x})^{\delta_{2}}\}\{1-F(L|\mathbf{x}%
)^{\delta_{3}}\}
\]
where $\delta_{1}$, $\delta_{2}$, and $\delta_{3}$ indicate either right-,
interval-, and left-censoring, respectively. The failure time $T$ is generated
from the distribution $F(t,\mathbf{x})=1-\exp\{-\Lambda_{0}(t)\exp
(x^{\intercal}\beta)\}$ and the baseline risk is simulated as $\Lambda
_{0}(t)=\log(1+t)+t^{1/2}$. The covariates are $\mathbf{x}=\{x_{1},x_{2}%
,x_{3},x_{4}\}$ where $x_{1},x_{2}\sim N(0,0.5^{2})$ and $x_{3},x_{4}\sim
$Bernoulli$(0.5)$. A censored interval is simulated by first generating
$Y\sim$Exponential$(1)$ and setting either $(L,R)=(Y,\infty)$ if $Y\leq T$ or
$(L,R)=(0,Y)$ if $Y>T$. The baseline risk is modeled as a six-parameter
I-spline, for a total of $10$ parameters to estimate. The sample size is set
to 2000 individuals.

Since the algorithms always converged on widely different objective values,
the allowed discrepancy was set to 2. The results are displayed in Figure
\ref{fig:cens} and appendix Table \ref{tab:EMPH}. ACX$^{3,2}$ was the top
performer in terms of CPU time and mappings used, followed by ACX$^{3,3,2}$.
Both outperformed DAAREM (2), ACX$^{2}$ and QNAMM (3). The non-accelerated EM
algorithm took orders of magnitudes longer to converge.

\subsubsection*{Alternating least squares for tensor rank decomposition}

ALS is an iterative method used in matrix completion, canonical tensor
decomposition and matrix factorization used in online rating systems, signal
processing, vision and graphics, psychometrics, and computational linguistics.
For canonical tensor decomposition, \cite{DeSterck2012} has designed an
iterative algorithm combining an ALS step and the nonlinear generalized
minimal residual method (N-GMRES) (\cite{Saad1986}). \cite{Riseth2019}
improved on the approach with a general-purpose acceleration method named
objective acceleration (O-ACCEL). 

The performance of O-ACCEL was compared to those of ACX and other
general-purpose acceleration using \cite{DeSterck2012}'s main test
specification. First introduced by \cite{Tomasi2006} and \cite{Acar2011}, the
test involves computing a 450-variable approximation of a three-way tensor of
size $50\times50\times50$ with collinearity and noise (see the code for
detail). For the tests, \cite{Riseth2019}'s MATLAB code was adapted and ACX,
QNAMM and DAAREM algorithms were written in MATLAB by the author. ACX was
implemented with $\underline{\sigma}=1$. O-ACCEL and DAAREM worked best with a
window of 10 past iterates and QNAMM worked well with a window of 5 past
iterates. Less performant algorithms like the N-GMRES or L-BFGS were not
tested. 

The results are summarized in Figure \ref{fig:tensor} and appendix Table
\ref{tab:tensor}. ACX$^{3,3,2}$ had a clear but moderate advantage, closely
followed by ACX$^{3,2}$ and QNAMM. O-ACCEL required fewer mappings on average
but its heavier computational burden made it slightly slower.%

\begin{figure}[tbp]%
\begin{minipage}{0.48\textwidth}%
\raisebox{-0cm}{\includegraphics[
height=7.0973cm,
width=7.0973cm
]%
{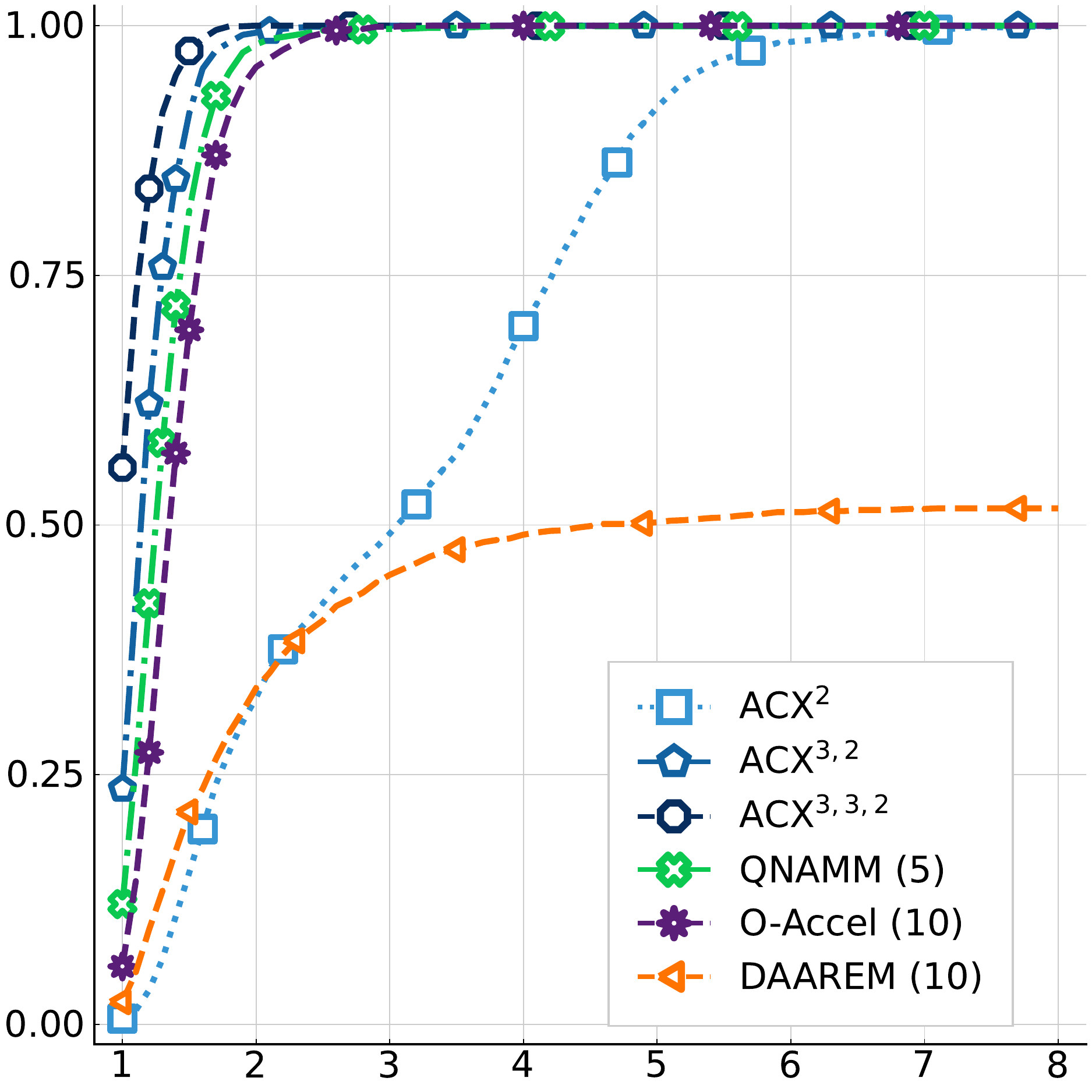}%
}
\caption{Performance profiles for the canonical vector decomposition\newline}%
\label{fig:tensor}
\end{minipage}%
\hfill%
\begin{minipage}{0.48\textwidth}
\raisebox{-0cm}{\includegraphics[
height=7.0973cm,
width=7.0973cm
]%
{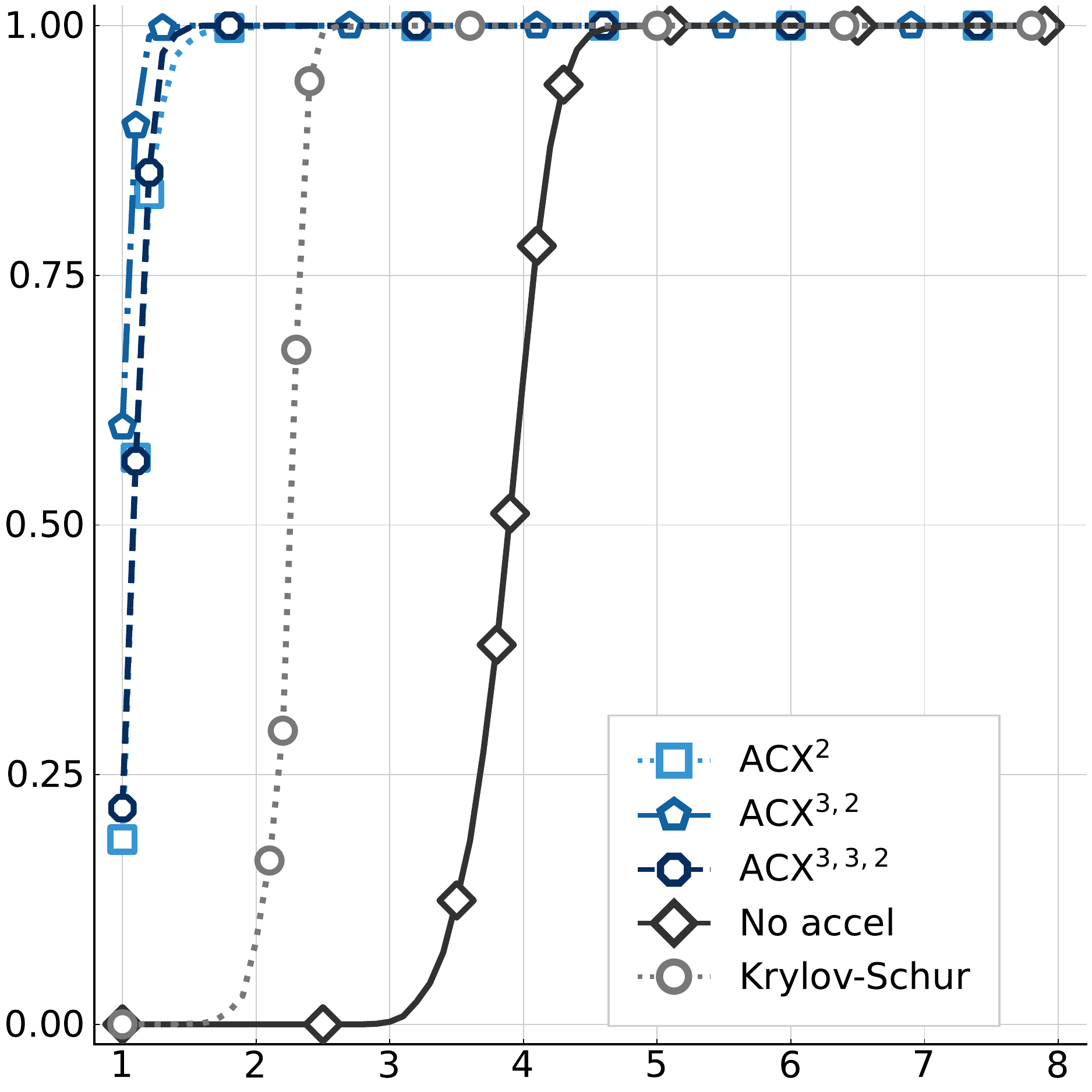}%
}
\caption
{Performance profiles for the power method for dominant eigenvalues \\}%
\label{fig:power}%
\end{minipage}%
\end{figure}%

\subsubsection*{The power method for dominant eigenvalues%
\label{sec:eigenvalues}%
}

Several big data applications with sparse structures involve computing a few
dominant eigenvalues for which iterative approaches like the power method are
clearly preferred. Given a diagonalizable matrix $Q$, the power method
computes the eigenvector associated with the dominant eigenvalue of $Q$ by
combining matrix-vector multiplications and rescaling: \bigskip%

\noindent
\textbf{The power method}

\smallskip%

\noindent
Start with a non-zero vector $x_{0}.$

\smallskip%

\noindent
Compute $x_{k+1}=\frac{Qx_{k}}{\left\Vert Qx_{k}\right\Vert _{\infty}}$ until convergence.

\bigskip

Under mild assumptions, the method generates a series of vectors $x_{k}$
converging to an eigenvector associated with the largest eigenvalue of $Q$ in
magnitude. Unfortunately, the method may be slow to converge if the absolute
value of ratio of the first and second largest eigenvalues is close to one.
\cite{Jennings1971} proposed accelerating it using the same multivariate
version of Aitken's $\Delta^{2}$ process suggested by \cite{Lemarechal1971}.
In the same vein, we expect for good acceleration using ACX schemes. 

For the test, a symmetric $1000\times1000$ sparse matrix was generated with
10\% non-zero elements drawn from $U[0,1]$. Random $U[0,100]$ were added to
the diagonal elements, creating a wide spectrum with similar magnitudes to the
first and second eigenvalues. 

ACX was compared with the unaccelerated power iteration and with the
Krylov-Schur algorithm, implemented in KrylovKit.jl (see \cite{Saad2000}). The
results are displayed in Figures \ref{fig:power} and in Table
\ref{tab:eigenvalues}. All ACX schemes performed well, taking around 3
milliseconds, while the Krylov-Schur took over 6 milliseconds and the
unaccelerated power iteration took over 10 milliseconds.

\subsubsection*{Alternating projections for high-dimensional fixed-effects
models}

Following \cite{VonNeuman1949}, finding the intersection between closed
subspaces has been an active area of research in numerical methods (see
\cite{Escalante2011} for a good treatment of the topic). The method of
alternating projections is a simple algorithm for finding such intersection.
Unfortunately, it can be arbitrarily slow to converge if the Friedrichs angle
between the subspaces is small, spawning a large literature on faster
algorithms. A noteworthy contribution was made by \cite{Gearheart1989} who
suggested a generalized Aitken's acceleration method (Scheme 3.4 in their
paper). Since their method actually corresponds to Lemar\'{e}chal's method
applied to the cyclic projection algorithm (alternating projections applied
sequentially to 2 or more subspaces), ACX constitutes a natural extension.

An interesting application to illustrate the potential of ACX and alternating
projections is a method suggested by \cite{Gaure2013}. In social sciences and
epidemiological studies, researchers often measure the impact of a few
variables on large samples of potentially time-varying observations while
controlling for stable unobserved effects. These could be worker effects, firm
effects, school effects, teacher effects, doctor effects, hospital effects,
etc. An example using public data is \cite{Head2010} who estimated the impact
of colonial relations on trade flows. Rather than reproducing their exact
results, consider the simple model%
\begin{equation}
\ln x_{ijt}=c_{ijt}\beta+d_{it}\gamma_{1}+d_{jt}\gamma_{2}+d_{ij}\gamma
_{3}+u_{ijt}\label{eq:trade}%
\end{equation}
where $x_{ijt}$ is the export volume from country of origin $i$ to country
destination $j$ in year $t$ (keeping only non-zero trade flows). Colonial
status is captured by $c_{ijt}$, a dummy variable that equals 1 if country $i$
is still a colony of country $j$ in year $t$, and $u_{ijt}$ represents
unobserved trade costs between the two countries at time $t$. The fixed
effects are dummies capturing observed or unobserved factors affecting trade
flows of the origin country at time $t$ ($d_{it}$), of the destination country
at time $t$ ($d_{jt}$) and stable characteristics of the exporter-importer
dyad such as common language, distance, etc. ($d_{ij}$). The parameter $\beta$
should capture the impact of being a colony on log exports to the metropolitan
state. Since colonial status is potentially correlated with the fixed effects,
omitting them from the model would most likely bias the estimate of $\beta$.

The full sample of trading countries totals $707,368$ observations and, more
importantly, the model contains $9569+9569+29,603=48,741$ fixed effects.
Needless to say, ordinary least squares estimation is impractical purely in
terms of memory. In contrast, the method proposed by \cite{Gaure2013} is fast
and lightweight. Define $x$ and $c$ as column vectors containing the $x_{ijt}$
and $c_{ijt}$, respectively, and $D=[D_{it},D_{jt},D_{ij}]$ as the matrix of
fixed effects. The method employs the Frisch-Waugh-Lovell theorem
(\cite{Frisch1933}, \cite{Lovell1963}) to estimate $\beta$ by regressing
$M_{D}x$ on $M_{D}c$, where $M_{D}=I-D(D^{\intercal}D)^{-1}D^{\intercal}$
projects onto the orthogonal complement of the column space of $D$. To avoid
computing $M_{D}$ directly, the method uses \cite{Halperin1962}'s Theorem 1:
$M_{D}=\lim_{k\rightarrow\infty}(M_{D_{it}}M_{D_{jt}}M_{D_{ij}})^{k}$. Since
the projections onto each set of fixed effects are equivalent to simple group
demeaning, the algorithm can be very fast and has become standard packages in
many programming languages.

The potential pitfalls of alternating projections remain, however. If the
panel is not well balanced and different sets of dummies are near collinear,
convergence may require many iterations and take longer than alternative
algorithms. In such cases, ACX acceleration could provide discernible benefits.%

\begin{table}[!htbp] \centering
\begin{minipage}{0.99\textwidth}%
\caption
{Performances for the trade flows regression with high-dimensional fixed effects}%
\begin{tabular}
[c]{lrrrrrr}\hline
&  &  &  &  &  & \vspace{-0.4cm}\\
Algorithm / Package &  & \multicolumn{2}{r}{Whole sample} &  &
\multicolumn{2}{r}{Partial sample}\\\cline{3-4}\cline{6-7}
&  &  &  &  &  & \vspace{-0.4cm}\\
&  & Maps & Time (sec.) &  & Maps & Time (sec.)\vspace{0.1cm}\\\hline
&  &  &  &  &  & \vspace{-0.4cm}\\
No acceleration &  & 21 & 0.41 &  & 82 & 0.80\\
ACX$^{2}$ &  & 16 & 0.31 &  & 26 & 0.32\\
ACX$^{3,2}$ &  & 16.5 & 0.31 &  & 25 & 0.29\\
ACX$^{3,3,2}$ &  & 16 & 0.29 &  & 24 & 0.28\\
FixedEffectModel (LSMR) (Julia) &  &  & 0.52 &  &  & 0.42\\
FELM (R) &  &  & 0.84 &  &  & 0.90\\
FixedEffectModel (Python) &  &  & 2.75 &  &  & 5.84\\
reghdfe (Stata) &  &  & 6.69 &  &  & 6.01\vspace{0.1cm}\\\hline
&  &  &  &  &  & \vspace{-0.4cm}\\
Observations &  & \multicolumn{2}{r}{707,368} &  &
\multicolumn{2}{r}{530,504\vspace{0.1cm}}\\\hline
&  &  &  &  &  & \vspace{-0.4cm}\\
\multicolumn{7}{p{15cm}}{Notes: Maps display the average number of iterations
to demean $x$ and $c$. For Python, the time only includes demeaning. Python
and R both use \cite{Gaure2013}'s method as well, but do not display the
number of mappings. Stata's reghdfe is too different for a meaningful
comparison of the number of mappings.\vspace{0.1cm}}\\\hline
\end{tabular}
\label{TableKey copy(2)}%
\label{tab:MAP}%
\end{minipage}%
\end{table}%

Gaure's method was used to estimate (\ref{eq:trade}), with and without ACX
acceleration, with a stopping criterion of $\left\Vert \Delta x\right\Vert
_{2}\leq10^{-8}$. For comparison, the same estimation was done with variety of
equivalent packages in Julia, R, Python, and Stata. The R and Python packages
implement Gaure's method. In Julia, FixedEffectModel.jl uses the LSMR
algorithm (\cite{Fong2011}) based on the Golub-Kahan bidiagonalization
(\cite{Golub1964}). For reghdfe, see \cite{Correia2016}. To test for the
impact of collinearity, the same model was estimated on the full sample and on
a subsample excluding the 25\% trading partners with the longest geographical
distance, making trading blocs more localized and the panel less balanced.

Table \ref{tab:MAP} displays the CPU time and average mapping for the
demeaning of $x$ and $c$, performed independently. For the whole sample, ACX
acceleration only provided a modest advantage over the unaccelerated
alternating projections. For the partial sample however, the unaccelerated
alternating projections was slower than LSMR, but ACX$^{3,3,2}$ reduced the
number of mappings and the compute time by 65\% and was again the fastest overall.

\section{Discussion}

This article has introduced new acceleration methods for fixed-point
iterations. By alternating between squared and cubic extrapolations, the ACX
schemes target specific error components and dynamically speed-up convergence
in subsequent iterations. Thanks to cycling, the extra computation needed for
cubic extrapolations is essentially free. For linear systems, ACX is Q-linear convergent.

Many optimization methods and fixed-point iteration accelerations store
information from past iterates, making them efficient in some contexts but
possibly less so in others. By only extrapolating from two or three mappings,
ACX schemes are remarkably fast, stable, and versatile. Applied to gradient
descent, they are competitive with the best nonlinear solvers. They also speed
up other fixed-point iterations like the EM algorithm, ALS, the power method,
and the method of alternating projections. These represent a small subset of
potential uses, which may extend to image processing, physics, and other big
data applications with sparse representations.

\bibliographystyle{erae}


\appendix

\part*{Appendix}

\section{Proofs}

\begin{proof}
[Proof of Lemma \ref{Lemma1}]From the binomial formula, express $(x+y)^{p}$ as
the sum of $x^{p}+y^{p}$ and an extra term:%
\begin{subequations}
\begin{align}
(x+y)^{p} &  =\sum_{j=0}^{p}\frac{p!}{j!\left(  p-j\right)  !}x^{p-j}%
y^{j}\nonumber\\
&  =x^{p}+y^{p}+\sum_{j=1}^{p-1}\frac{p!}{j!(p-j)!}x^{p-j}y^{j}\nonumber\\
&  =x^{p}+y^{p}+\sum_{j=0}^{p-2}\frac{p!}{(j+1)!(p-j-1)!}x^{p-j-1}%
y^{j+1}\nonumber\\
&  =x^{p}+y^{p}+pxy(R_{1})\label{eq:R1}%
\end{align}
where
\end{subequations}
\[
R_{1}=\sum_{j=0}^{p-2}\frac{p-1}{(j+1)(p-1-j)}\frac{(p-2)!}{j!(p-2-j)!}%
x^{p-2-j}y^{j}.
\]
Consider the case of a general extra term:%
\[
R_{i}=\sum_{j=0}^{p-2i}\frac{i(p-i)}{(j+i)((p-i)-j)}\frac{(p-2i)!}%
{j!(p-2i-j)!}x^{p-2i-j}y^{j}\text{.}%
\]
Let us consider different cases. By direct calculation, if $p$ is even and
$i=p/2$, $R_{i}=1$.\ If $p$ is odd and $i=(p-1)/2$, $R_{i}=x+y$. If $1\leq
i<\left\lfloor p/2\right\rfloor $, using the binomial formula:%
\begin{align*}
R_{i} &  =(x+y)^{p-2i}+\sum_{j=0}^{p-2i}\left(  \frac{i(p-i)}{(j+i)((p-i)-j)}%
-1\right)  \frac{(p-2i)!}{j!(p-2i-j)!}x^{p-2i-j}y^{j}\\
R_{i} &  =(x+y)^{p-2i}-\sum_{j=0}^{p-2i}\frac{j}{j+i}\frac{p-2i-j}{p-i-j}%
\frac{(p-2i)!}{j!(p-2i-j)!}x^{p-2i-j}y^{j}\text{.}%
\end{align*}
If $j=0$ or $j=p-2i$ then $\frac{j}{j+i}\frac{p-2i-j}{p-i-j}=0$. We may
therefore write%
\[
R_{i}=(x+y)^{p-2i}-\sum_{j=1}^{p-2i-1}\frac{j}{j+i}\frac{p-2i-j}{p-i-j}%
\frac{(p-2i)!}{j!(p-2i-j)!}x^{p-2i-j}y^{j}%
\]
and rewrite the sum as%
\begin{align*}
R_{i} &  =(x+y)^{p-2i}-\sum_{j=0}^{p-2i-2}\frac{j+1}{j+1+i}\frac
{p-2i-j-1}{p-i-j-1}\frac{(p-2i)!}{(j+1)!(p-2i-j-1)!}x^{p-2i-j-1}y^{j+1}\\
R_{i} &  =(x+y)^{p-2i}-\left[
\begin{array}
[c]{c}%
xy\frac{(p-2i)(p-2i-1)}{(i+1)(p-(i+1))}\times\\
\sum_{j=0}^{p-2(i+1)}\frac{(i+1)(p-(i+1))}{(j+i+1)(p-(i+1)-j)}\frac
{(p-2(i+1))!}{j!(p-2(i+1)-j)!}x^{p-2(i+1)-j}y^{j}%
\end{array}
\right]  \\
R_{i} &  =(x+y)^{p-2i}-xy\frac{(p-2i)(p-2i-1)}{(i+1)(p-i-1)}R_{i+1}\text{.}%
\end{align*}
By recursively substituting $R_{i+1}$ back in \ref{eq:R1}, we may rewrite the
whole expression using a summation for $i=1$ to $\left\lfloor p/2\right\rfloor
$ and recover (\ref{eq:Lemma1}).
\end{proof}

\newpage

\section{Tables of numerical results%
\label{app:results}%
}%

\FloatBarrier
%

\begin{table}[!htbp]%
\begin{center}%
\begin{minipage}{0.8\textwidth}%
\caption{Average performances: Unconstrained Rosenbrock}%
\label{tab:Rosenbrock}%
\begin{tabular}
[c]{lrrrrr}\hline
&  &  &  &  & \vspace{-0.4cm}\\
Algorithm & Grad. evals & Obj. evals & Time (ms) & Conv. & Minimum\vspace
{0.1cm}\\\hline
&  &  &  &  & \vspace{-0.4cm}\\
ACX$^{2}$ & 907.9 & 11.0 & 14.23 & 1.00 & 2.00E-15\\
ACX$^{3,2}$ & 720.7 & 11.0 & 10.37 & 1.00 & 2.00E-15\\
ACX$^{3,3,2}$ & 596.7 & 11.0 & 8.83 & 1.00 & 1.00E-15\\
L-BFGS & 759.1 & 759.1 & 15.75 & 1.00 & 1.20E-14\\
N-CG & 929.7 & 1855.4 & 25.56 & 1.00 & 3.00E-15\vspace{0.1cm}\\\hline
&  &  &  &  & \vspace{-0.4cm}\\
\multicolumn{6}{l}{Note: 1000 parameters.\vspace{0.1cm}}\\\hline
\end{tabular}%
\end{minipage}%
\end{center}%
\end{table}%
%

\begin{table}[!htbp]%
\begin{center}%
\begin{minipage}{0.8\textwidth}%
\caption{Average performances: Constrained Rosenbrock}%
\label{tab:RosenbrockConstr}%
\begin{tabular}
[c]{lrrrrr}\hline
&  &  &  &  & \vspace{-0.4cm}\\
Algorithm & Grad. evals & Obj. evals & Time (ms) & Conv. & Minimum\vspace
{0.1cm}\\\hline
&  &  &  &  & \vspace{-0.4cm}\\
ACX$^{2}$ & 458.0 & 6.0 & 9.18 & 1.00 & 1.98E+02\\
ACX$^{3,2}$ & 358.6 & 6.0 & 7.08 & 1.00 & 1.98E+02\\
ACX$^{3,3,2}$ & 408.2 & 6.0 & 8.14 & 1.00 & 1.98E+02\\
L-BFGS & 4815.0 & 4815.0 & 295.95 & 1.00 & 1.98E+02\\
N-CG & 3574.1 & 5006.7 & 266.78 & 1.00 & 1.98E+02\vspace{0.1cm}\\\hline
&  &  &  &  & \vspace{-0.4cm}\\
\multicolumn{6}{l}{Notes: 1000 parameters. For ACX, $\omega=0.999$%
.\vspace{0.1cm}}\\\hline
\end{tabular}%
\end{minipage}%
\end{center}%
\end{table}%
%

\begin{table}[!htbp]%
\begin{center}%
\begin{minipage}{0.8\textwidth}%
\caption{Average performances: Logistic regression}%
\label{tab:logistic}%
\begin{tabular}
[c]{lrrrrr}\hline
&  &  &  &  & \vspace{-0.4cm}\\
Algorithm & Grad. evals & Obj. evals & Time (sec.) & Conv. & Minimum\vspace
{0.1cm}\\\hline
&  &  &  &  & \vspace{-0.4cm}\\
ACX$^{2}$ & 53.2 & 5.3 & 29.85 & 1.00 & 612.1403\\
ACX$^{3,2}$ & 51.8 & 5.3 & 25.14 & 1.00 & 612.1403\\
ACX$^{3,3,2}$ & 51.8 & 5.3 & 25.55 & 1.00 & 612.1403\\
L-BFGS & 64.8 & 64.8 & 164.48 & 1.00 & 612.1403\\
N-CG & 61.8 & 94.2 & 257.92 & 1.00 & 612.1403\vspace{0.1cm}\\\hline
&  &  &  &  & \vspace{-0.4cm}\\
\multicolumn{6}{l}{Note: 100 parameters.\vspace{0.1cm}}\\\hline
\end{tabular}%
\end{minipage}%
\end{center}%
\end{table}%
%

\begin{table}[!htbp]%
\begin{center}%
\begin{minipage}{0.8\textwidth}%
\caption{Average performances: EM algorithm for Poisson admixture}%
\label{tab:mixtures}%
\begin{tabular}
[c]{lrrrrr}\hline
&  &  &  &  & \vspace{-0.4cm}\\
Algorithm & Maps & Obj. evals & Time (ms) & Conv. & Minimum\vspace
{0.1cm}\\\hline
&  &  &  &  & \vspace{-0.4cm}\\
ACX$^{2}$ & 102.1 & 0.0 & 0.53 & 1.00 & 1989.9459\\
ACX$^{3,2}$ & 56.0 & 0.0 & 0.31 & 1.00 & 1989.9459\\
ACX$^{3,3,2}$ & 61.1 & 0.0 & 0.33 & 1.00 & 1989.9459\\
QNAMM (3) & 113.0 & 108.0 & 0.62 & 1.00 & 1989.9459\\
No acceleration & 2524.0 & 0.0 & 11.35 & 1.00 & 1989.9459\\
DAAREM (2) & 67.5 & 84.6 & 1.15 & 1.00 & 1989.9459\vspace{0.1cm}\\\hline
&  &  &  &  & \vspace{-0.4cm}\\
\multicolumn{6}{p{13cm}}{Notes: 3 parameters, 0.95\% of draws rejected for
discrepant results. ACX implemented with stabilization mapping and
$\omega=0.8$.\vspace{0.1cm}}\\\hline
\end{tabular}%
\end{minipage}%
\end{center}%
\end{table}%
%

\begin{table}[!htbp]%
\begin{center}%
\begin{minipage}{0.8\textwidth}%
\caption
{Average performances: proportional hazards regression with interval censoring}%
\label{tab:EMPH}%
\begin{tabular}
[c]{lrrrrr}\hline
&  &  &  &  & \vspace{-0.4cm}\\
Algorithm & Maps & Obj. evals & Time (sec.) & Conv. & Minimum\vspace
{0.1cm}\\\hline
&  &  &  &  & \vspace{-0.4cm}\\
ACX$^{2}$ & 2562.1 & 0 & 4.84 & 0.885 & 952.5416754\\
ACX$^{3,2}$ & 333.3 & 0 & 0.64 & 0.991 & 952.3066159\\
ACX$^{3,3,2}$ & 394.3 & 0 & 0.74 & 0.993 & 952.2839194\\
QNAMM (3) & 1840.3 & 1223.6 & 3.43 & 0.995 & 952.0988158\\
DAAREM (5) & 752.2 & 1211.1 & 1.62 & 0.985 & 952.3447906\vspace{0.1cm}\\\hline
&  &  &  &  & \vspace{-0.4cm}\\
\multicolumn{6}{p{13cm}}{Notes: 10 parameters, 17.55\% of draws rejected for
discrepant results.\vspace{0.1cm}}\\\hline
\end{tabular}%
\end{minipage}%
\end{center}%
\end{table}%
%

\begin{table}[!htbp]%
\begin{center}%
\begin{minipage}{0.8\textwidth}%
\caption{Average performances: canonical tensor decomposition}%
\label{tab:tensor}%
\begin{tabular}
[c]{lrrrrr}\hline
&  &  &  &  & \vspace{-0.4cm}\\
Algorithm & Maps & Obj evals & Time (sec.) & Conv. & Minimum\vspace
{0.1cm}\\\hline
&  &  &  &  & \vspace{-0.4cm}\\
ACX$^{2}$ & 228.3 & 57.4 & 1.74 & 1.00 & 0.0741\\
ACX$^{3,2}$ & 90.0 & 18.4 & 0.68 & 1.00 & 0.0741\\
ACX$^{3,3,2}$ & 82.9 & 15.9 & 0.62 & 1.00 & 0.0741\\
QNAMM (5) & 88.8 & 81.8 & 0.74 & 1.00 & 0.0741\\
O-Accel (10) & 56.6 & 56.6 & 0.8 & 1.00 & 0.0741\\
DAAREM (10) & 556.5 & 606.5 & 1.15 & 0.52 & 0.0741\vspace{0.1cm}\\\hline
&  &  &  &  & \vspace{-0.4cm}\\
\multicolumn{6}{p{13cm}}{Notes: 450 parameters, 6.9\% of draws rejected for
discrepant results. One ALS iteration is registered as one mapping.\vspace
{0.1cm}}\\\hline
\end{tabular}%
\end{minipage}%
\end{center}%
\end{table}%
%

\begin{table}[!htbp]%
\begin{center}%
\begin{minipage}{0.5\textwidth}%
\caption{Average performances: dominant eigenvalues}%
\label{tab:eigenvalues}%
\begin{tabular}
[c]{lrrr}\hline
&  &  & \vspace{-0.4cm}\\
Algorithm & Maps & Time (ms) & Conv.\vspace{0.1cm}\\\hline
&  &  & \vspace{-0.4cm}\\
ACX$^{2}$ & 29.9 & 3.07 & 1.00\\
ACX$^{3,2}$ & 28.0 & 2.83 & 1.00\\
ACX$^{3,3,2}$ & 30.1 & 3.03 & 1.00\\
No acceleration & 110.4 & 10.64 & 1.00\\
Krylov-Schur & 53.7 & 6.14 & 1.00\vspace{0.1cm}\\\hline
&  &  & \vspace{-0.4cm}\\
\multicolumn{4}{l}{Note: 1000 parameters.\vspace{0.1cm}}\\\hline
\end{tabular}%
\end{minipage}%
\end{center}%
\end{table}%
%

\FloatBarrier

\section{Software and hardware used%
\label{app:Software}%
}

For the CUTEst and the alternating projections applications, each problem was
run five times if the median time was over 0.1 seconds, and 100 times if it
was below 0.1 seconds. Then the median time is reported. All tests in Julia
were run once before recording time to exclude compile time. See the code for detail.

The main software used for the numerical experiments were Julia v.1.6.1
(\cite{bezanson2017}), FixedEffectModels.jl v1.6.1, Optim.jl v1.3.0
(\cite{mogensen2018optim}), CUTEst.jl v0.11.1, KrylovKit 0.5.3, MATLAB R2018a,
Stata 13, reghdfe (Stata package), Python v3.8.5, FixedEffectModel v0.0.2
(Python), R v4.0.4, and felm (lfe v 2.8-6, R).

All computations were single-threaded, done on HP ZBooks 15 with Intel Core
i7-4900MQ CPUs with 2.80GHz and 32 Go of RAM. All were done on Ubuntu 20.04,
except the alternating projections using MATLAB on Windows 10.
\end{document}